\documentclass[final]{siamltex}

\usepackage[ansinew]{inputenc}
\usepackage[T1]{fontenc}
\usepackage{graphicx}
\usepackage{amsmath}
\usepackage{amssymb}
\usepackage{amsfonts,latexsym,amssymb}
\usepackage{comment}
\usepackage{color}
\usepackage{epstopdf}
\usepackage{float}
\usepackage{mathrsfs}
\usepackage{hyperref}

\title{ISS Lyapunov strictification via observer design and integral action control for a Korteweg-de Vries equation}

\author{Ismaila Balogoun\thanks{Ismaila Balogoun and Swann Marx are with LS2N, Ecole Centrale de Nantes and CNRS UMR 6004, Nantes, France. {\tt\small E-mail: \{ismaila.balogoun,swann.marx\}@ls2n.fr}.}\and Swann Marx$^{\star}$ \and  Daniele Astolfi \thanks{Daniele Astolfi is with  Universit\'e Lyon 1 CNRS UMR 5007 LAGEPP, France . {\tt\small daniele.astolfi@univ-lyon1.fr}.}}

\usepackage{array}

\usepackage{float}
\usepackage{tikz}
\usepackage{pgf,tikz}
\usepackage{graphicx}
\usepackage{enumitem}

\usepackage{comment}

\usepackage{indentfirst}
\usepackage{enumitem}

\usepackage{ulem}
\usepackage{comment}

\catcode`\@=11
\def\downparenfill{$\m@th\braceld\leaders\vrule\hfill\bracerd$}
\def\overparen#1{\mathop{\vbox{\ialign{##\crcr\crcr
\noalign{\kern0.4ex}
\downparenfill\crcr\noalign{\kern0.4ex\nointerlineskip}
$\hfil\displaystyle{#1}\hfil$\crcr}}}\limits}
\catcode`\@=12

\newcommand{\overbar}[1]{\mkern1.5mu\overline{\mkern-1.5mu#1\mkern-1.5mu}\mkern 1.5mu}
\def\NN{{\mathbb N}}    % set of natural number
\def\RR{{\mathbb R}}    % field of real number
\def\VV{{\mathcal V}}

\def\cC{{\mathcal C}}
\def\cA{{\mathcal A}}
\def\cS{{\mathcal S}}
\def\cM{{\mathcal M}}
\def\cT{{\mathcal T}}

\def\L{{L^2}}

\begin{document}

\maketitle
%\thispagestyle{empty}
%\pagestyle{empty}

%%%%%%%%%%%%%%%%%%%%%%%%%%%%%%%%%%%%%%%%%%%%%%%%%%%%%%%%%%%%%%%%%%%%%%%%%%%%%%%%
\begin{abstract}
The article deals with the output regulation of a
nonlinear Korteweg-de-Vries (KdV) equation subject to a distributed disturbance. The control input and the regulated output are located at the boundary. To achieve this objective, we follow a Lyapunov approach. {To this end} inspired by a strictification methodology
recently introduced in the finite-dimensional context, we construct an ISS-Lyapunov functional for the KdV equation thanks to the use of an observer which is designed following the backstepping approach. Then, thanks to this Lyapunov functional, we apply the forwarding approach in order to solve the desired output regulation problem.

\end{abstract}
\begin{keywords}
  Input-to-state stability, integral controller, Korteweg-de Vries equation, backstepping,
  observer, regulation, forwarding.
\end{keywords}

%%%%%%%%%%%%%%%%%%%%%%%%%%%%%%%%%%%%%%%%%%%%%%%%%%%%%%%%%%%%%%%%%%%%%%%%%%%%%%%

\section{Introduction}

This paper
deals with the output regulation of a Korteweg-de Vries (KdV) equation. The KdV equation is a mathematical model of waves on shallow water surfaces (see e.g., \cite{cerpa2014} for a survey). Such an equation has been studied in \cite{rosier,chapouly,coron2020small} in the controllability context, in \cite{cerpacoron,lucoron,marxcerpa,tang2015stabilization,swanncpa1} in terms of stabilization, and in \cite{tang2018asymptotic,chu2015asymptotic} where some asymptotic analysis of the equilibrium point
coinciding with the origin are given. We may also mention \cite{marx2018stability,marx2019stability} where 
\textit{input-to state stability} (ISS, in short) properties are obtained via feedback stabilization in presence of a saturated damping (we refer to \cite{prieurmazenc,jacob2020noncoercive,kafnemer2021weak} or the recent survey \cite{prieur} for the characterization of \color{black}ISS Lyapunov functional in the infinite-dimensional context).
Roughly speaking,  the output regulation problem
consists in designing a
feedback-law such that the output converges asymptotically towards a desired  reference and such that 
disturbances are rejected, possibly in spite of some ``small'' model uncertainties. Following the celebrated \textit{internal-model principle}, a solution
to such a problem
exists when references and disturbance (denoted generically as exosignals) are generated by a known
autonomous dynamical system (denoted as exosystem), and a copy of such a system is embedded
in the controller dynamics, see, e.g. \cite{francis1975internal,paunonen2010internal}. 
A well known example is the use of integral action
for tracking and rejecting constant references and disturbances.

Output regulation is an old topic in the finite-dimensional context, 
but many results remain to be found
in the context of nonlinear systems
(see e.g., \cite{astolfi2017integral,giaccagli2021sufficient} for  recent results in this field), and many further research lines
have to be followed when dealing with time-varying references. See, for instance, \cite{astolfi2021repetitive} 
 where a finite-dimensional system is regulated by adding a transport equation for the case of periodic exosignals.
 For infinite-dimensional systems, even if one can mention some old results such as \cite{de2003boundary}, the topic is still very active. 
 A generalization of internal-model principle has
been proposed in \cite{paunonen2010internal}, but the use of integral action to achieve output regulation in the 
presence of constant references/perturbations for infinite dimensional systems has been initiated 
early in \cite{pohjolainen1982robust}. 
Since then, several methods to design an integral action have been developed
for linear dynamics following,
\color{black}for instance, a spectral approach in \cite{pohjolainen1985robust,xu1995robust,paunonen2010internal}, by using operator and semi-group methods in \cite{logemann1997low,xu2014multivariable}, based on frequency domain methods with Laplace transform in \cite{bastin2015stability,coron2015feedback} or by relying on Lyapunov techniques in \cite{lhachemi2020pi}, \cite{dos2008boundary,trinh2016multivariable}. We may also mention \cite{deutscher2017finite,deutscher2017output} which propose to regulate an output towards time-varying references that are generated by a known linear dynamical system or \cite{liard2022boundary} which extends the sliding mode methodology for hyperbolic systems to reject time-varying disturbances.
In the context of nonlinear PDEs, 
we recall also the works 
\cite{natarajan2016approximate,huhtala2021approximate,zhang2020local}
\color{black}

Among all these techniques, in this article,  we are particularly interested in Lyapunov techniques. Indeed, such a methodology has been proved to be efficient to deal with nonlinear systems. Among these techniques, we aim at using the forwarding methodology that has been first introduced for 
finite-dimensional systems 
in cascade form \cite{mazenc1996adding,astolfi2017integral} and then extended to some hyperbolic systems \cite{terrand2019adding} in the regulation context, and to abstract systems \cite{marx2021forwarding} in the stabilization context. In \cite{terrand2019adding}, it is shown that a \textit{strict} Lyapunov functional\footnote{Strict Lyapunov functionals are Lyapunov functionals whose time-derivative is bounded by a negative function depending on the full-state \cite[\S 2.1]{fmazenc}.} is needed for open-loop stable systems that we aim at regulating. In other words, before adding an integral action, we should be able to show that a strict Lyapunov functional for the open-loop dynamics does exist
(or can be obtained after employing a preliminary stabilizing state-feeedback, see, e.g. \cite{astolfi2017integral} in the finite dimensional context). Such Lyapunov functionals are known for hyperbolic systems \cite{bastin2021input}, but it is not the case for the KdV equation. In addition to the existence of this Lyapunov functional, some ISS properties are needed to apply the forwarding method.

In the perspective of addressing the output regulation problem
for the KdV equation, 
we first establish some \color{black}new results that may have their own interest.
In particular, we  study the following 
(nonlinear) KdV equation
\begin{equation}\label{eq_kdv_nonlin}
\left\{
\begin{array}{ll}
w_t+w_{x} +w_{xxx}+ww_x=d_1(t,x)\,, &\\
w(t,0)=w(t,L)=0\,, &\\
w_{x}(t,L)=d_2(t)\,,& \\
w(0,x)=w_0(x)\,,
%y(t)=w_x(t,0)
\end{array}
\right.
\end{equation}
where $(t,x)\in\mathbb{R}_+\times [0,L]$, 
$L>0$,  $d_1$ and $d_2$ denote external inputs that might be seen, for instance, as disturbances, 
and its associated linearized {dynamics} around the origin 
%equilibrium point $0$ 
described by
\begin{equation}\label{eq_kdv_lin}
\left\{
\begin{array}{ll}
w_t+w_{x} +w_{xxx}=d_1(t,x)\,, &\\
w(t,0)=w(t,L)=0 \,, &\\
w_{x}(t,L)=d_2(t) \,,& \\
w(0,x)=w_0(x)\,.
%y(t)=w_x(t,0)
\end{array}
\right.
\end{equation}
where $(t,x)\in \mathbb{R}_+ \times [0,L]$. 
We show that the KdV equations \eqref{eq_kdv_nonlin} and \eqref{eq_kdv_lin} 
satisfy an ISS property with respect to the disturbances $d_1,d_2$
by explicitly constructing a strict Lyapunov functional. 
Note that there is no systematic method to build strict Lyapunov functionals either for nonlinear ordinary differential equations or (linear or nonlinear) partial differential equations. However, in many situations, a \textit{weak} Lyapunov functional, i.e., a Lyapunov functional whose time derivative is nonpositive, exists. Often, it also coincides with the energy of the system.
It is however difficult to deduce any quantitative robustness properties from a weak Lyapunov functional, and in particular, ISS properties 
cannot be generically obtained
 from such functions. For this reason, in the finite-dimensional context, a lot of attention has been put in the \textit{strictification} of  weak Lyapunov functions, namely the 
 conception
  of systematic procedures to modify  a weak Lyapunov function in order to make it strict. See, for instance, \cite{fmazenc,praly2019}.  To the best of our knowledge, in the infinite-dimensional context, such an approach has been applied  only  to certain classes of hyperbolic systems \cite{prieurmazenc}.

The first contribution of this paper, that might be seen thus of independent interest with respect to the 
context of output regulation, is the construction of an 
ISS-Lyapunov functional for our KdV equation via a strictification procedure. 
The methodology we propose 
is inspired on \cite{praly2019} and is based on the design of an observer, 
which is also a new result in the KdV equation context and therefore
consists in the second main contribution of this article.
Let us illustrate it. Consider system \eqref{eq_kdv_nonlin}
with no inputs, namely $d_1=d_2=0$. A formal computation shows that the time derivative of the energy $E$ defined as
\begin{equation}
\label{eq:energy}
E(w):= \int_0^L w(t,x)^2 dx
\end{equation}
 yields along solutions 
 \begin{equation}
\label{eq:energy-weak}
    \dot{E}(w):=\frac{d}{dt}\int_0^L w(t,x)^2 dx = -  |w_x(t,0)|^2.
\end{equation}
These computations are sufficient to establish
that the origin is Lyapunov stable
but not to conclude stronger properties (such as 
asymptotic stability or an ISS property if 
we re-introduce the effect of the disturbances 
in the computation of  the derivative of the energy along the trajectories of 
\eqref{eq_kdv_nonlin}).
% Although Lyapunov stability can be concluded, nothing can be said about the attractivity and ISS-properties of the plant.
In other words, the energy $E$ is a weak-Lyapunov functional. Since $w_x(t,0)$ is an exactly observable output as soon as $L\notin \mathcal{N}$ with 
$$
\mathcal{N}:=\left\{2\pi\sqrt{\tfrac{k^2+kl+l^2}{3}}:k,l\in\NN\right\},
$$ 
then, following \cite{praly2019}, our strategy consists in designing an observer with the output $w_x(t,0)$. 
Such an observer is obtained by using the \textit{backstepping} approach
(see, e.g., \cite{KrsticAndrey}) and the  Fredholm operator (see, e.g., \cite{lucoron} or \cite{gagnon2020fredholm}).
The proposed observer differs from the works in 
 \cite{marx2014output,marxcerpa,tang2015stabilization} in the same context of KdV equations because a different measured output is considered.
Finally, by combining the Lyapunov functional derived from the observer analysis and the energy $E$, we  obtain a strict Lyapunov functional, that will be used to establish the desired ISS properties  for systems \eqref{eq_kdv_nonlin} and \eqref{eq_kdv_lin} with respect to the inputs $d_1$ and $d_2$. 

Finally, the third contribution of this article consists in addressing the output regulation problem for constant perturbations and references.
We suppose that a control is acting at the boundary $w_x(t,L)$ of the  KdV equations \eqref{eq_kdv_nonlin} and \eqref{eq_kdv_lin}
and that we want to regulate the output $w_x(t,0)$ to a desired reference $r$.
As a consequence, we extend the plant with an integral action processing the error 
$w_x(t,0)-r$ and we show how to design an output-feedback law. The gain of the controller is obtained via the forwarding technique which is employed to construct a strict Lyapunov functional built upon the ISS Lyapunov functional obtained in the first part of this article. 
Global stability properties
are  established for the linear model 
\eqref{eq_kdv_lin} while only local ones
are proved for the nonlinear one 
\eqref{eq_kdv_nonlin}. Note that, in both cases, we prove pointwise convergence of the tracking error i.e $\lim_{t\to \infty} \vert w_x(t,0)-r\vert=0$.  Also, note that the results of 
\cite[Theorems 1, 2]{terrand2019adding} cannot be used of the shelf 
for the linear model \eqref{eq_kdv_lin}
because in our article
we consider a control input 
acting at the boundary (to be more precise, with an unbounded operator).
Nevertheless, since we know an ISS-Lypaunov function we can 
apply the proposed methodology, similarly to what
has been done for the hyperbolic equations in 
\cite[Theorem 3]{terrand2019adding}.
In the context of output regulation of 
nonlinear PDEs as in  
\eqref{eq_kdv_nonlin}, there exist very few results. Let us mention for instance \cite{zhang2020local} which studies quasilinear hyperbolic systems. 
We recall also \cite{natarajan2016approximate,huhtala2021approximate} in which the local problem is solved for regular linear operators perturbed by nonlinearities 
satisfying a Lipschitz condition.
Note, however, that these results
do not directly apply to the KdV nonlinear model
because the nonlinearity $ww_x$ is not Lipschitz
in the right space.
In this article, we are able to solve 
the local regulation problem  for \eqref{eq_kdv_nonlin} thanks to the strict Lyapunov functional that we established.

This paper is organized as follows. In Section~\ref{sec:prob}, we formulate the problem and state the results about the construction of the ISS Lyapunov functional. In Section~\ref{sec:obs}, an observer is designed using a Fredholm operator. Section~\ref{sec:proof} contains the proofs of the ISS results of the KdV equations under consideration. Section~\ref{Add} states and proves some regulation results for the KdV equation. Finally, Section~\ref{sec:con} collects concluding remarks and discuss some remaining open problems.

\medskip

\textbf{Notation:} Set $\mathbb{R}_+=[0,\infty)$. The term $w_t$  stands for the partial derivative of the function $w$ with respect to $t$. The term $w_x$ (resp. $ w_{xx}, w_{xxx}$) stands for the first (resp. second and third) order partial derivative of the function $w$ with respect to $x$.
When a function $V$ (resp. $M$) depends only the time variable $t$ (resp. the space variable $x$), we use the notation $\dot V(t) = \frac{d}{dt} V(t)$ (resp. $M^\prime(x):= \frac{d}{dx} M(x)$). 
 The functional space $L^2(0,L)$ denotes the set of (Lebesgue) measurable functions $f$ such that $\int_0^L |f(x)|^2 dx<+\infty$. The associated norm is $\Vert f\Vert_{L^2}^2:=\int_0^L |f(x)|^2 dx.$ 
We define the functional space  $C^2([0,T])$ as  the class of continuous functions on $[0,T]$, which have continuous derivatives of order two on $[0,T]$, the  functional spaces $H^k(0,L)$. %Finally, we set $C^2_0([0,T]):=\{d\in C^2([0,T]) : d(0)=0\}$, 
% $ H^3_L(0,L):=\Big\lbrace w\in H^3(0,L): w(0)=w(L)=w^\prime(L)=0\Big\rbrace$, $ H^3_0(0,L):=\Big\lbrace w\in H^3(0,L): w(0)=w(L)=w^\prime(0)=0\Big\rbrace$.
For any $p\in [1,\infty]$, we use the standard notation 
 $W^{1,p}(0,L)$  for the Sobolev space defined as  $W^{1,p}(0,L):=\lbrace u\in L^p(0,L):\: \dot{u}\in L^p(0,L)\rbrace$. 
\medskip

\section{Construction of an ISS Lyapunov functional}
%\section{Problem statement and main result }
\label{sec:prob}

The objective of this section is to 
study the ISS properties of the KdV models
\eqref{eq_kdv_nonlin} and \eqref{eq_kdv_lin}
and to establish the existence of a strict
ISS-Lypaunov functional. 
The proof of the main result is postponed to 
Section~\ref{sec:proof}.
Furthermore, as mentioned in the introduction, the proposed 
ISS-Lyapunov functional  will be used in the sequel 
in order to design an output feedback 
integral action controller, see Sections~\ref{sec_reg_linear} and \ref{sec_reg_nonlinear}. Note that we will not provide further discussions on the well-posedness of \eqref{eq_kdv_lin} and \eqref{eq_kdv_nonlin}, since it is not the main topic of this paper. Interested readers may refer to \cite{rosier,corcre,bona2003nonhomogeneous} for more information on this issue. We just emphasize on the fact that, when looking at regular solutions, we will consider initial conditions in the space
\begin{equation}
\label{def-H3L}
H^3_L(0,L):=\lbrace w\in H^3(0,L):\: w(0)=w(L)=0,\: w^\prime(L) = d_2(0)\rbrace,
\end{equation}
with $d_2$ being the perturbation entering at the boundary condition in \eqref{eq_kdv_nonlin} or \eqref{eq_kdv_lin}.
In this case, for any $T>0$, solutions $w$ to \eqref{eq_kdv_nonlin} or \eqref{eq_kdv_lin} belong to the functional space $C(0,T;H^3(0,L))\cap C^1(0,T;L^2(0,L))$ and satisfy, for all $t\in [0,T]$, the additional compatibility conditions $w(t,0)=w(t,L)=~0$, $w_x(t,L)=d_2(t)$.

Next, we state the following definition of
input-to-state stability 
for systems \eqref{eq_kdv_nonlin} and \eqref{eq_kdv_lin}.

\medskip

\begin{definition}\label{df}
System \eqref{eq_kdv_nonlin} (resp.  \eqref{eq_kdv_lin}) is said to be (exponentially)  input-to-state stable (ISS),  if there exist positive constants $c_0$, $c_1$, $c_2$, $\mu$, such that any solution $w\in C^0(\RR_+;L^2(0,L))\cap L^2(\RR_+;H^1(0,L))$   to \eqref{eq_kdv_nonlin} (respectively to \eqref{eq_kdv_lin})  satisfies for all $t\geq0$
\begin{equation}
    \Vert w(t,\cdot)\Vert_{L^2}\leq c_0e^{-\mu t}\Vert w_0\Vert_{L^2}   + c_1\int_{0}^t e^{-\mu(t-s)}\|d_1(s,\cdot)\|_{L^2}ds  + c_2\int_{0}^t e^{-\mu(t-s)}|d_2(s)|ds , 
\label{def-iss}
\end{equation}
for any initial condition $w_0\in L^2(0,L)$,  
   $d_1\in L^2([0,t];L^2(0,L))$ and  $d_2\in L^2(0,t)$.
 Furthermore, if there exists $\delta>0$ such that
  \eqref{def-iss} holds only  with
  $w_0, d_1, d_2$ satisfying
  $$
    \Vert w_0 \Vert_{\L} + 
\lim_{t\to\infty}  \int_{0}^t e^{-\mu(t-s)}\Big(\|d_1(s,\cdot)\|_{L^2} + 
|d_2(s)| \Big)ds 
\,\leq \, 3\delta
  $$
  then the system \eqref{eq_kdv_nonlin} (resp. \eqref{eq_kdv_lin}) is said to be locally 
 (exponentially)  input-to-state stable (LISS).
\end{definition}

\medskip

In the literature, the definition 
\eqref{def-iss}
is related to the notion of the 
``Fading Memory Input-to-State Stability'', see e.g \cite[Chapter 7]{karafyllis2019input}, 
due to the presence of weighting
exponential functions used in the  norms characterizing the gain of the signals $d_1$ and $d_2$. Thus, with some  abuse of language, we call it Input-to-State Stability in this paper.
Also, it is important to underline that such a definition allows to consider 
a large class of disturbances $d_1,d_2$, which includes, among others, constant and periodic signals. %For the definition of local ISS, note that more regular initial conditions are considered. This comes from the results that we have obtained in the sequel of the paper. Moreover, regular initial conditions are needed to ensure pointwise convergence of the output towards desired reference.

In general, proving the ISS property defined above needs the knowledge of the trajectories of the system, which is  not an easy task. Therefore, in practice, ISS Lyapunov functionals are used to prove the desired ISS properties. To this end, we recall the result in  \cite[Theorem 3]{prieur}, showing that  the existence of an ISS Lyapunov functional is sufficient to establish the ISS properties of Definition \ref{df}.

Before stating the definition of such Lyapunov functionals, we recall now which type of derivatives we are going to use in this article. Indeed, for any Lyapunov functional $V$ for 
solutions to 
\eqref{eq_kdv_lin} or \eqref{eq_kdv_nonlin}, one has the following equality:
\begin{equation}
\label{eq:frechet}
\dot  V (w)= \frac{d}{dt} V(w) = D_V(w)w_t, 
\end{equation}
where $D_V(w)$ denotes the Fr\'echet derivative (see for instance \cite[Definition A.5.33]{curtain2020introduction} for the definition ). The proof of this equality follows the same path than the one given in \cite[Lemma 11.2.5]{curtain2020introduction}.  For instance, this means that the time derivative along solutions to \eqref{eq_kdv_lin} of $E(w) = \Vert w\Vert^2_\L$ can be computed as
\begin{equation}
\dot E(w) = 2\int_0^L (-w_x-w_{xxx} + d_1)w \, dx,
\end{equation}
and, for time derivative along solutions to \eqref{eq_kdv_nonlin}:
\begin{equation}
   \dot  E(w) = 2\int_0^L (-w_x-w_{xxx} - ww_x + d_1)w \, dx.
\end{equation}
showing that the time does not play any role when using the Fr\'echet derivative. This is why the time will disappear when differentiating Lyapunov functionals in the rest of the paper.

We are now in position to state the following definition of ISS Lyapunov functional.

\medskip

\begin{definition}
 A function $V: L^2(0,L)\to\RR$ is said to be an  exponentially \textit{ISS Lyapunov functional} for the system \eqref{eq_kdv_nonlin} (resp. \eqref{eq_kdv_lin}), if there exist positive constants $\underline \alpha, \bar \alpha, \alpha, \sigma_1, \sigma_2 $ such that:
\begin{enumerate}
    \item[(i)] For all $w\in L^2(0,L)$,
\begin{equation}
    \label{eq:Lyap-sandwich}
    \underline\alpha\Vert w\Vert_\L^2\leq V(w)\leq \bar\alpha\Vert w\Vert_\L^2\,.
\end{equation}
    \item[(ii)]  The time derivative of $V$ along the trajectories of \eqref{eq_kdv_nonlin} (resp. \eqref{eq_kdv_lin}) satisfies
    \end{enumerate}
\begin{equation}
    \label{eq:Lyap-ISS}
 \dot{V}(w)\leq-\alpha\Vert w\Vert_\L^2+\sigma_1\Vert d_1\Vert_\L^2
        + \sigma_2\vert d_2\vert^2\,.
\end{equation}
for any $w\in \L(0,L)$, $d_1\in \L(0,L)$ and $d_2\in \RR$.
If there exists $\delta>0$ such that 
(ii) holds only if 
$\Vert w\Vert_{\L} + \Vert d_1\Vert_\L + |d_2|
\leq 3\delta$
 then $V$ is said to be a locally exponentially  ISS Lyapunov functional for the system \eqref{eq_kdv_nonlin}. 
 \end{definition}
 
 \medskip
 
\color{black}
As explained in the introduction, for any $L\notin \mathcal{N}$, 
the energy function  in $L^2$-norm
defined in \eqref{eq:energy} 
is a weak Lyapunov functional for the system
\eqref{eq_kdv_nonlin} (resp. \eqref{eq_kdv_lin})  in  view of \eqref{eq:energy-weak}.
Indeed, on the the right hand side  of the inequality, we have a function
which depends only on a part the state $w(t,x)$, i.e., $w_x(t,0)$.
 From \eqref{eq:energy-weak}, one can deduce that the origin of the system of \eqref{eq_kdv_nonlin} with $d_1=d_2=0$ is Lyapunov stable. 
 In order to show also the exponential stability properties of the origin, 
 one can follow \cite[Proposition 3.3]{rosier}, by using the 
 fact  that $w_x(t,0)$ is exactly observable as soon as $L\notin \mathcal{N}$: indeed, using the related observability inequality, and integrating \eqref{eq:energy-weak} between $0$ and $T$, exponential stability can be established as illustrated in \cite[\S 4.1.]{cerpa2014}. 
 However,  nothing can be easily said in the  presence
 of disturbances. As a consequence, 
 in order to show the desired ISS properties of the  system
\eqref{eq_kdv_nonlin} (resp. \eqref{eq_kdv_lin}),  we follow a different approach here: we aim at constructing a strict ISS Lyapunov functional, which is a new result, to the best of our knowledge. Using the observability of the output $w_x(t,0)$, we can follow the methodology described in \cite{praly2019} and that can be decomposed as follows.  First, we design an observer for the output $w_x(t,0)$.  Then, we consider the sum of the Lyapunov functional coming from the latter observer design and the natural energy, and  we prove that this sum boils down to be a strict Lyapunov functional. Finally, thanks to this strict Lyapunov functional, we deduce ISS properties for systems \eqref{eq_kdv_nonlin} and \eqref{eq_kdv_lin}. These properties are written more precisely in the following theorem, which is our first main result.

\medskip

\begin{theorem}
\label{theo}
Suppose that $L\notin \mathcal{N}$. Then, there exists a functional $W:  L^2(0,L)\to  \mathbb{R}_+$ such that, 
the function $V(w) := W(w)+ E(w)$
with $E$ being the energy in $L^2$-norm defined in  
\eqref{eq:energy}, is 
\begin{enumerate}
    \item[(a)] an exponentially ISS Lyapunov functional for the system \eqref{eq_kdv_lin};
    \item[(b)] a locally exponentially ISS Lyapunov functional for the system \eqref{eq_kdv_nonlin}.
\end{enumerate}

Moreover, the functional $W$ is given by
$W(w):=  \Vert\Pi(w)\Vert^2_{L^2}$ with $\Pi$ being a continuous
linear operator from $L^2(0,L)$ to $L^2(0,L)$ with a continuous inverse.
\end{theorem}

\medskip
The proof of Theorem~\ref{theo} is postponed to Section~\ref{sec:proof}.
In particular, in the next section, we will first show how to design an ISS observer 
for the linearized system \eqref{eq_kdv_lin} by means of the output $w_x(t,0)$. 
The proposed design is based on the backstepping method, see, e.g., \cite{KrsticAndrey} and on 
the Fredholm transformation, see, e.g., \cite{gagnon2020fredholm,lucoron}. Then,
in Section~\ref{sec:proof},
we will use the ISS-Lyapunov functional associated to such an observer to build the functional $W$ claimed in the statement of 
Theorem~\ref{theo}.

The following result will be also useful when dealing with the regulation problem. It is an ISS result for a perturbed version of \eqref{eq_kdv_nonlin} with non-constant (small) coefficients. Its proof is omitted for compactness since it follows the same path used in  the proof of Theorem~\ref{theo}, item (b). 

\medskip

\begin{corollary}\label{corollary:ISS-nonlineaire}
Suppose $L\notin \mathcal{N}$. There exists positive real numbers $\bar a, \bar b$ such that,
for any $a\in C([0,L])$, $b\in C^1([0,L])$ satisfying 
$\Vert a\Vert_{\infty}\leq \bar a$ and 
$\Vert b\Vert_{W^{1,\infty}}\leq \bar b$,
the Lyapunov function $V$ established 
in Theorem~\ref{theo} 
is a locally exponentially ISS Lyapunov functional  for the following system
\begin{equation}\label{eq_kdv_nonlinear_perturbed}
\left\{
\begin{array}{lll}
 w_t+  w_{x} + w_{xxx}+ w w_{x}=
 a(x)w +  b(x)w_x \,, &\qquad & (t,x)\in  \RR_+\times [0,L]\,,\\
 w(t,0)= w(t,L) = 0\,, && t\in \RR_+\,,\\
 w_{x}(t,L)=d_2(t)\,, && t\in \RR_+\,, \\
 w(0,x)= w_0(x), && x\in [0,L]\,.
\end{array}
\right.
\end{equation}
Moreover, it is an exponential ISS Lyapunov functional
for the linearized dynamics of \eqref{eq_kdv_nonlinear_perturbed}, i.e. in absence of the term $ww_x$.
\end{corollary}

\section{Observer design for a Linear KdV equation}
\label{sec:obs}
In this section,  we design an observer for 
the linear KdV equation \eqref{eq_kdv_lin} with $y(t)=w_x(t,0)$  defined as the output function. 
 In particular, we consider the following system
\begin{equation}
 \left\{
\begin{array}{lll}
w_t+w_{x} +w_{xxx}=d_1 \,, &\qquad & (t,x)\in \RR_+\times [0,L] \,,\\
w(t,0)=w(t,L)=0 \,, && t \in \RR_+ \,,\\
w_{x}(t,L)= d_2(t) \color{black}\,, && t \in \RR_+ \,, \\
w(0,x)=w_0(x) \,, & &x \in [0,L] \,, \\
y(t)=w_{x}(t,0) \,, && t \in \RR_+ \,,
\end{array}
\right.
\label{eqo}
\end{equation}
and we design an observer 
with a distributed correction term
of the form
\begin{equation}
    \left\{
\begin{array}{lll}
\widehat{w}_t+\widehat{w}_{x} +\widehat{w}_{xxx}+ p(x)[y(t)
-\widehat{w}_x(t,0)]=0 \,,  &\qquad& (t,x)\in \RR_+\times [0,L] \,, \\
\widehat{w}(t,0)=\widehat{w}(t,L)=0 \,, && t \in \RR_+ \,,\\
\widehat{w}_{x}(t,L)=0 \,, && t \in \RR_+ \,,\\
\widehat{w}(0,x)= \widehat w_0(x) \,, && x \in [0,L] \,,
\end{array}
\right.
\label{eqob}
\end{equation}
where $p$ is an output injection gain to be designed.  Note that the well-posedness of system \eqref{eqob} can be proved by following the same approach as in \cite{marx2014output}.
We define now the estimation error coordinates as follows
$$
\widehat w \mapsto \widetilde w : = w- \widehat w
$$
mapping system 
\eqref{eqob} into 
\begin{equation}
\left\{
\begin{array}{lll}
\widetilde{w}_t+\widetilde{w}_{x} +\widetilde{w}_{xxx}
- p(x)\widetilde{w}_x(t,0)=d_1 \,, &\qquad & (t,x)\in \RR_+\times [0,L] \,,\\
\widetilde{w}(t,0)=\widetilde{w}(t,L)=0\,, && t \in \RR_+ \,,\\
\widetilde{w}_x(t,L)=d_2(t) \,, && t \in \RR_+ \,,\\
\widetilde{w}(0,x)=\widetilde w_0(x)\,, && x \in [0,L] \,.
\end{array}
\right.
\label{eqereur*}
\end{equation}
The objective of this section is to show that 
the gain $p$ can be selected so that to guarantee  
the system \eqref{eqereur*} to be ISS 
with respect to the disturbances $d_1, d_2$. 
This, in turns, guarantees the convergence 
of the solutions of the observer 
\eqref{eqob} towards the trajectories 
of the observed plant
\eqref{eqo} in the unperturbed case 
 ($d_1=0$, $d_2=0$), 
 and desirable bounded-input bounded-output
 properties otherwise.
 This is established in the next
 theorem claiming the existence of an 
 ISS-Lyapunov functionals for the system 
 \eqref{eqereur*} under an appropriate choice of the 
 function $p$.
 \color{black}
 \medskip
% In nominal conditions ($d_1=0$, $d_2=0$), the origin of the 
% estimation error $\widetilde w$ with error-dynamics 
% \eqref{eqereur*} has to be 
% exponentially stable so that to guarantee the convergence of the estimate $\widehat w$ to state $w$.
% In the presence of disturbances, we would like 
% the error-dynamics $\widetilde w$
% to possess some bounded-input bounded-output properties so that for small perturbations 
% the estimated state $\widehat w$ doesn't deviate too much from the real value of the state $w$.
% In other words,  we are interested in guaranteeing 
% the estimation error dynamics to be ISS with respect to the perturbations $d_1,d_2$.
% This is shown in the next result 
% showing that it is always possible to select the gain $p$ so that to satisfy the desired
% ISS-properties
% of the observer \eqref{eqob}.\\

\begin{theorem}
\label{theo:ISS-Lyap-obs}
Suppose that $L\notin \mathcal{N}$. For any $\lambda>0$,  there exist 
a non-zero function $p\in L^2(0,L)$, 
a Lyapunov functional $U:L^2(0,L)\to\RR$
and  positive constants $\underline c, \bar c,  \varrho_1, \varrho_2$ satisfying the following properties.
\begin{itemize}
    \item[(i)]  For all $\widetilde w\in L^2(0,L)$
     \begin{equation}
         \label{eq:Usandwich}
         \underline c \Vert \widetilde w\Vert_{L^2}^2 \leq  U(\widetilde w) \leq \bar c\Vert \widetilde w\Vert_{L^2}^2 \,.
     \end{equation}
     \item[(ii)] The time derivative of $U$ along the trajectories of 
    \eqref{eqereur*}  satisfies, for all $w\in \L(0,L)$,  $d_1\in \L(0,L)$ and $d_2 \in \RR$,
    \begin{equation}
         \label{eq:Uder}
 \dot{U}(\widetilde w)\leq -  \lambda U(\widetilde w) +  \varrho_1\Vert d_1\Vert_\L^2
+  \varrho_2|d_2|^2  .
     \end{equation}
\end{itemize}
Moreover, the  functional  $U$ is given by 
$U(w) := \Vert\overline \Pi^{-1}(w)\Vert^2_{L^2}$, 
with $\overline \Pi$ being a continuous linear operator
from $\L(0,L)$ to $\L(0,L)$ with continuous inverse.
\end{theorem}
 \color{black}
 \medskip
 
\begin{proof} 
The proof of Theorem \ref{theo:ISS-Lyap-obs} is divided into two parts. The first step consists in proving the existence  of $p\in \L(0,L)$ such that the origin of  \eqref{eqereur*}, in the unperturbed case $d_1=0$, $d_2=0$, is exponentially stable. The second step is to show the existence of  a Lyapunov functional $U$  satisfying  the inequalities \eqref{eq:Usandwich} and \eqref{eq:Uder}.

Let us start the proof   of the first step. Inspired by \cite[equation (1.8)]{lucoron}, consider the change of coordinates
\begin{equation}
    \label{eq:change_of_coord_gamma}
\widetilde w \mapsto \gamma := \overline \Pi^{-1} \widetilde w
\end{equation}
where the function  $\overline \Pi$ is defined thanks to  
the following Fredholm integral transformation 
\begin{equation}
\label{eq:change_of_coord}
\widetilde{w}(x)=\overline\Pi(\gamma)(x):= \gamma(x)-\int^L_0 P(x,z)\gamma(z)dz \,, 
\end{equation}
for all $x\in [0,L]$, where $\widetilde{w}$ satisfies \eqref{eqereur*} with $d_1=0$ and $d_2=0$, $P$ is a function to be defined and $\gamma$ is the solution to the following system
\begin{equation}
\left\{
\begin{array}{lll}
\gamma_t+\gamma_{x} +\gamma_{xxx}+ \lambda\gamma=0\,,&\qquad  &(t,x)\in \RR_+\times [0,L],\\
\gamma(t,0)=\gamma(t,L)=\gamma_x(t,L)=0\,,  &&{t\in\mathbb{R}_+}\,,\\
\gamma(0,x)=\gamma_0(x)\,, && x\in [0,L],
\end{array}
\right.
\label{eqgama}
\end{equation}  with $\lambda>0$.
Note that using an integration by parts and the boundary conditions of \eqref{eqgama},
one immediately obtains
\color{black}
$$
\frac{d}{dt}\int^L_0\vert \gamma(t,x)\vert^2dx\leq -2\lambda\int^L_0\vert \gamma(t,x)\vert^2dx
$$
from which it is straightforward to deduce the exponential stability in 
the $L^2$-norm of $\gamma$.
As a consequence, the main idea of the proof consists in 
selecting  the function $p$  such that \eqref{eq:change_of_coord} holds. To do so, we need to find the kernel $P$ such that $\widetilde{w}(t,x)=\overline\Pi(\gamma)(t,x)$ satisfies \eqref{eqereur*} when $d_1=0$ and $d_2=0$. Furthermore, we have also to ensure that the corresponding transformation is invertible and continuous.
To this end, we first 
formally differentiate 
with   respect to the time and with respect to the space 
the change of coordinates \eqref{eq:change_of_coord}.
We obtain the following identities
\begin{align}
\label{eq:change_of_coor_t}
\widetilde{w}_t(t,x) &=  \displaystyle
\gamma_t(t,x)  +  \int^L_0  P(x,z)\Big(\lambda\gamma(t,z)+\gamma_z(t,z)+\gamma_{zzz}(t,z)\Big)dz\,,
\\
\label{eq:change_of_coor_x}
\widetilde{w}_x(t,x)& =\gamma_x(t,x)- \int^L_0 P_x(x,z)\gamma(t,z)dz\,,
\\
\label{eq:change_of_coor_xxx}
\widetilde{w}_{xxx}(t,x)&=\gamma_{xxx}(t,x)-\int^L_0 P_{xxx}(x,z)\gamma(t,z)dz\,,
\end{align}
in which \eqref{eq:change_of_coor_t} has been obtained 
by using the $\gamma$-dynamics in \eqref{eqgama}.
\color{black}
After some integrations by parts, \eqref{eq:change_of_coor_t} 
gives
\begin{align}
\label{eq:wt}
    \widetilde{w}_t(t,x)=& \, \gamma_t(t,x) -P(x,0)\gamma(t,0)+P(x,L)\gamma(t,x)+P(x,L)\gamma_{xx}(t,L) -P(x,0)\gamma_{xx}(t,0)+P_z(x,0)\gamma_x(t,0)\notag\\
    &\displaystyle - \int^L_0 \Big( -\lambda P(x,z)+P_z(x,z)+P_{zzz}(x,z) \Big)\gamma(t,z)dz-P_z(x,L)\gamma_x(t,L)+P_{zz}(x,L)\gamma(t,L)
    \notag\\&-P_{zz}(x,0)\gamma(t,0)\,.
\end{align}
Then, by adding on both sides 
the  terms 
$\widetilde{w}_x$, 
$\widetilde{w}_{xxx}$ 
and 
$-p(x)\widetilde{w}_x(t,0)$
and using  \eqref{eqereur*}, \eqref{eqgama}
and the previous identities 
\eqref{eq:change_of_coor_x}, 
\eqref{eq:change_of_coor_xxx}, 
we further obtain\color{black}
$$
  \begin{array}{l}
    \widetilde{w}_t(t,x)+\widetilde{w}_{x}(t,x) +\widetilde{w}_{xxx}(t,x)- p(x)\widetilde{w}_x(t,0) = 
    \\ \displaystyle 
    \quad = \gamma_t(t,x)+\gamma_{x}(t,x) +\gamma_{xxx}(t,x)+ \lambda\gamma(t,x)
  -\int^L_0 \Big( -\lambda P+P_z+P_{zzz}+P_{xxx}+P_x\Big)\gamma(t,z)dz
  \\
 \quad \phantom{=}\displaystyle -\lambda\gamma(t,x) +P(x,L)\gamma_{xx}(t,L)+P_z(x,0)\gamma_{x}(t,0)
-P(x,0)\gamma_{xx}(t,0)-p(x)\bigg[\gamma_x(t,0)-\int^L_0 P_x(0,z)\gamma(t,z)dz \bigg]
  \end{array}  
$$
% \begin{align*}
%     \widetilde{w}_t(t,x)+\widetilde{w}_{x}(t,x) +\widetilde{w}_{xxx}(t,x)- p(x)\widetilde{w}_x(t,0) &= \gamma_t(t,x)+\gamma_{x}(t,x) +\gamma_{xxx}(t,x)+ \lambda\gamma(t,x)
%  \\
%  & \displaystyle
%   -\int^L_0 \Big( -\lambda P+P_z+P_{zzz}+P_{xxx}+P_x\Big)\gamma(t,z)dz
%   \\
%   & \displaystyle -\lambda\gamma(t,x) +P(x,L)\gamma_{xx}(t,L)+P_z(x,0)\gamma_{x}(t,0)
%  \\ 
%  & -P(x,0)\gamma_{xx}(t,0)-p(x)\bigg[\gamma_x(t,0)-\int^L_0 P_x(0,z)\gamma(t,z)dz \bigg],
% \end{align*}
where some arguments are omitted for compactness when clear from the context. 
Then, using the identity
$$
-\lambda\gamma(t,x) =  \int_0^L \lambda  \delta(x-z)\gamma(t,z) dz\,,
$$
where $\delta(x - z)$ 
denotes the Dirac measure on the diagonal of the square $[0, L] \times [0, L]$, 
the previous equation gives

\begin{equation}
\begin{array}{l}
     \widetilde{w}_t(t,x)+\widetilde{w}_{x}(t,x) +\widetilde{w}_{xxx}(t,x)- p(x)\widetilde{w}_x(t,0)
     \\ 
  \displaystyle \quad   = \gamma_t(t,x)+\gamma_{x}(t,x) +\gamma_{xxx}(t,x)+ \lambda\gamma(t,x)
 -\int^L_0\Big(-\lambda P+P_z 
+P_{zzz}+P_x + P_{xxx}  -\lambda \delta(x-z)\Big)\gamma(t,z) dz
\\ 
\quad \phantom{=} \displaystyle -P(x,0)\gamma_{xx}(t,0)+P(x,L)\gamma_{xx}(t,L)+
p(x)\!\!\int^L_0\!\! P_x(0,z)\gamma(t,z)dz
 -\gamma_{x}(t,0)\big[p(x)-P_z(x,0) \big].
 \end{array}\label{eq:conditions}
\end{equation}
% \begin{align}\label{eq:conditions}\notag
%      \widetilde{w}_t(t,x)+\widetilde{w}_{x}(t,x) +\widetilde{w}_{xxx}(t,x)- p(x)\widetilde{w}_x(t,0)&= \gamma_t(t,x)+\gamma_{x}(t,x) +\gamma_{xxx}(t,x)+ \lambda\gamma(t,x)
%  \\ \notag
%  &
%  \displaystyle
%  -\int^L_0\Big(-\lambda P+P_z 
% +P_{zzz}+P_x + P_{xxx}  -\lambda \delta(x-z)\Big)\gamma(t,z) dz
% \\ \notag
% &\displaystyle -P(x,0)\gamma_{xx}(t,0)+P(x,L)\gamma_{xx}(t,L)+
% p(x)\!\!\int^L_0\!\! P_x(0,z)\gamma(t,z)dz
% \\
% &\displaystyle
%  -\gamma_{x}(t,0)\big[p(x)-P_z(x,0) \big].
% \end{align}
From equation \eqref{eq:conditions}, we finally obtain the following conditions
for the functions $P$ and $p$.
\smallskip
\begin{enumerate}[label=(\textit{\alph*})]
    \item The  identity  $-\lambda P+P_z+P_{zzz}+P_x+P_{xxx}=  \lambda \delta(x-z)$ is satisfied for all $(x,z)\in [0,L]\times [0,L]$.
    \item The boundary conditions  $P(x,0)=P(x,L)=P_x(0,z)=0$ are satisfied for all $(x,z)\in [0,L]\times [0,L]$.
    \item An appropriate choice of $p$ is given by
$p(x):=P_z(x,0)$, for all $x\in [0,L]$.
\end{enumerate}
Moreover, note also that the following.
\begin{enumerate}[label=(\textit{\alph*})]  \setcounter{enumi}{3}
    \item By setting $x=0$ and $x=L$ in \eqref{eq:change_of_coord},  we need: $P(0,z)=P(L,z)=0$ for all $z\in [0,L]$.
    \item By setting $x=L$ in
    \eqref{eq:change_of_coor_x},  we need: $P_x(L,z)=0$ for all $z\in [0,L]$.
\end{enumerate}
 Therefore, 
collecting the conditions $(a)$-$(e)$, 
we impose  the function $P$ to satisfy the following PDE
\begin{equation}
    \left\{
\begin{array}{l} 
-\lambda P+P_z+P_{zzz}+P_x
+P_{xxx}=\lambda \delta(x-z) \,,\\
P(x,0)=P(x,L)=0\,,\\
P(L,z)=P(0,z)=0 \,,  \\
P_x(L,z)=P_x(0,z)=0 \,,
\end{array}
\right.
\label{noyau}
\end{equation}
where $(x,z) \in [0,L]\times[0,L]$ and $\delta(x - z)$ denotes the Dirac measure on the diagonal of the square $[0, L] \times [0, L]$.
Now, in order to show the existence of a solution to 
\eqref{noyau}, let us
make the following change of variable:
$$
\begin{pmatrix}
z
\\
x
\end{pmatrix}
\mapsto
\begin{pmatrix}
\bar x 
\\
\bar z
\end{pmatrix}
:= 
\begin{pmatrix}
L- z
\\
L-x
\end{pmatrix},
$$ 
and define $G(\bar{x},\bar{z}) := -P(x,z)$. From 
\eqref{noyau} it is obtained
\begin{equation}
    \left\{
\begin{array}{l}
\lambda G+G_{\bar{z}}+G_{\bar{z}\bar{z}\bar{z}}+G_{\bar{x}}+G_{\bar{x}\bar{x}\bar{x}}=\lambda \delta(\bar{x}-\bar{z})\,,\\
G(\bar{x},0)=G(\bar{x},L)=0\,,\\
G(L,\bar{z})=G(0,\bar{x})=0 \,,  \\
G_{\bar{z}}(\bar{x},0)=G_{\bar{z}}(\bar{x},L)=0 \,,
\end{array}
\right.
\label{noyau1}
\end{equation}
with $(\bar x,\bar z)$ belonging to  $[0,L]\times[0,L]$.
 Note that in  \cite[Lemma 2.1]{lucoron}, it has been proved that, for any $L  \notin \mathcal{N}$, the system \eqref{noyau1} admits a unique solution $G\in  H_0^1((0, L) \times (0, L))$.
 Therefore, we can conclude that the kernel $P$ exists. Then according to \cite[Lemma 3.1]{lucoron}, the transformation $\overbar\Pi$ is invertible and continuous on $ L^2(0,L)$ and its inverse is also
continuous. 
As a consequence, 
we have shown that, for an appropriate
choice of the function 
 $p\in \L(0,L)$,
the system 
\eqref{eqereur*} is transformed into 
the system 
\eqref{eqgama}
via a linear change of coordinates 
which is invertible with a continous inverse.
Since the origin of system \eqref{eqgama}
is exponentially stable, 
we conclude that so is the origin of 
\eqref{eqereur*} in the non-perturbed case 
(i.e., $d_1=0$, $d_2=0$).
Note also that $p$ is non-zero. Indeed, if $p=0$ then, in view of the condition (\textit{c}), we would have $P_z(x,0)=0$. Therefore, the system \eqref{noyau} would have seven boundary conditions. But then,  because of the degree of the first equation of \eqref{noyau}, the system \eqref{noyau} would have no solution.
This concludes the first part of the proof.

We want now to prove the existence of  a Lyapunov functional  which satisfies  the inequalities \eqref{eq:Usandwich} and \eqref{eq:Uder}
in presence of $d_1,d_2$. 
To this end, we choose the following candidate Lyapunov function $U:\L(0,L)\to \RR$ \begin{equation}
    \label{eq:Udef}
   U(w) := \Vert \overline\Pi^{-1}(w) \Vert^2_\L
\end{equation}
Since $\overline \Pi^{-1}$ exists, then $U$ is well defined in $L^2(0,L)$. Moreover, according to the continuity of $\overline\Pi^{-1}$ and  $\overline\Pi$  in $L^2(0,L)$,
there exist two positive constants $\underline c$ and $\bar c$ satisfying
inequality \eqref{eq:Usandwich} for all $w\in \L(0,L)$.
\color{black}
Note that the function 
$w\in  L^2(0,L) \mapsto U(w)\in \mathbb{R}_+$ is equivalent to the standard norm on the space 
$ L^2(0,L)$ according to \eqref{eq:Usandwich}. 
It only remains to prove that $U$ satisfies the inequality 
\eqref{eq:Uder}. To this end, we show inequality 
\eqref{eq:Uder} for $\widetilde w_0\in H^3_L(0,L)$, $d_2\in C^2([0,T])$ and $d_1\in C^1([0,T],L^2(0,L))$. The result follows for all $\widetilde w_0\in L^2(0,L)$,  $d_1\in L^1([0,T];L^2(0,L))$
and  $d_2\in L^2(0,T)$, by a standard density argument similar to the one used in \cite[Lemma 1]{marx2017cone}. Now, consider again the transformation defined in 
\eqref{eq:change_of_coord_gamma}, \eqref{eq:change_of_coord}. 
Similar computations can be used to show that its inverse transformation 
is defined by
\begin{equation}
    \label{inverse}
\gamma(x):=\overline\Pi^{-1}(\widetilde{w})(x)= \widetilde{w}+\int^L_0 Q(x,z)\widetilde{w}(z)dz \,, 
\end{equation} where $Q\in H_0^1((0, L) \times (0, L))$ is now the 
solution of the following system  
\begin{equation}
    \left\{
\begin{array}{l} 
\lambda Q+Q_z+Q_{zzz}+Q_x
+Q_{xxx}=\lambda \delta(x-z) \,,\\
Q(x,0)=Q(x,L)=0\,,\\
Q(L,z)=Q(0,z)=0 \,,  \\
Q_x(L,z)=0 \,,
\end{array}
\right.
\label{noyau2} 
\end{equation} 
and satisfies
$p(x)+\int^L_0p(z) Q(x,z)dz=Q_z(x,0)$ for all $x\in [0,L]$.
 Now, consider the solution $\widetilde{w}$ of system \eqref{eqereur*}
with $d_1,d_2$ possibly different from zero. Then, applying 
the change of coordinates $\gamma= \overline\Pi^{-1}(\widetilde{w})$
defined in 
 \eqref{inverse}, \eqref{noyau2}, we obtain
\begin{equation}
\left\{
\begin{array}{lll}
\gamma_t+\gamma_{x} +\gamma_{xxx}+ \lambda\gamma=\overline \Pi^{-1}(d_1)+ Q_z(x,L)d_2,\quad & &(t,x)\in \RR_+\times [0,L],\\
\gamma(t,0)=\gamma(t,L)= 0 \,,
& \qquad &{t\in\mathbb{R}_+},\\
\gamma_x(t,L)=d_2(t), & \qquad &{t\in\mathbb{R}_+},\\
\gamma(0,x)=\gamma_0(x) \,, && x\in [0,L] \,.
\end{array}
\right.
\label{eqgama_d}
\end{equation} 
The derivative of $U$ along the trajectory of \eqref{eqereur*}, 
or equivalently on 
the trajectory of \eqref{eqgama_d},
yields
\begin{align}\label{eq:tempUder}\notag
    \dot U(w)  = &-2 \displaystyle \int_0^L \gamma \Big(\gamma_{xxx} + \gamma_x +\lambda\gamma-\overline \Pi^{-1}(d_1)-Q_z(x,L)d_2\Big)dx  \\ \notag\displaystyle
 =&  -2 \lambda\int_0^L \vert \gamma\vert^2dx +2\int_0^L\gamma_x\gamma_{xx}dx +2\int_0^L \overline \Pi^{-1}(d_1)\gamma dx +2d_2\int_0^L Q_z(x,L)\gamma dx
\\
\leq &
\displaystyle -2\lambda\Vert \gamma\Vert_\L^2
+2\left|\int_0^L \overline \Pi^{-1}(d_1)\gamma dx\right| +d_2^2 -\gamma_x(0)^2 + 2\left|d_2\int_0^L Q_z(x,L)\gamma dx\right|\,,
\end{align}
where, in the second equation, we have used an integration by parts to compute 
$$
2\int_0^L\gamma_x \gamma_{xx}dx 
 =  \Big[ \gamma_x^2 (x)
 \Big]_0^L 
  = d_2^2 -\gamma_x(0)^2 \,.
$$
 Using first
 Cauchy-Schwarz's inequality and then Young's inequality
 $2ab \leq \nu a^2 +\tfrac{1}\nu b^2$, for any $\nu>0$, 
 from \eqref{eq:tempUder}
 we finally obtain
\begin{align*}
\dot U(w) \leq &-2\lambda\Vert \gamma\Vert_\L^2 + 2\Vert \gamma\Vert_\L\Vert \overline\Pi^{-1}(d_1)\Vert_\L+ 2\vert d_2\vert \Vert \gamma\Vert_\L\Vert Q_z(\cdot,L)\Vert_\L
+ |d_2|^2 
\\
 \leq & -\lambda\Vert \gamma\Vert_\L^2 +  \tfrac2\lambda\Vert \overline \Pi^{-1}(d_1)\Vert_\L^2
+ \left(1+\frac{2}{\lambda}\Vert Q_z(\cdot,L)\Vert_\L^2\right)|d_2|^2  \,.
\end{align*} 
Using the inequality \eqref{eq:Usandwich} on the term depending on $d_1$, we finally obtain 
\begin{equation}
    \label{eq:tempUder2}
\dot U(w)  \leq -\lambda\Vert \gamma\Vert_\L^2 +  \frac{2\bar c}{\lambda}\Vert d_1\Vert_\L^2
+ \left(1+\frac{2}{\lambda}\Vert Q_z(\cdot,L)\Vert_\L^2\right)|d_2|^2  
\end{equation}
showing the  inequality \eqref{eq:Uder} with 
$\varrho_1=\frac{2\bar c}{\lambda}, \varrho_2=1+\frac{2}{\lambda}\Vert Q_z(\cdot,L)\Vert_\L^2$.
This completes the proof.
\end{proof}

From the existence of the ISS Lyapunov functional 
established in Theorem~\ref{theo:ISS-Lyap-obs}, one can immediately 
deduce the following property for the observer
\eqref{eqob}. 

\medskip

\begin{corollary}\label{corollary:iss-observer}
For any $\lambda>0$,  there exists a function $p\in L^2(0,L)$  such that the observer 
\eqref{eqob} is an ISS
exponential convergent observer
for system \eqref{eqo}
with convergence rate $\lambda$,  namely, 
there exist some $c_0,c_1,c_2>0$ 
such that 
the following inequality holds
\begin{equation}
    \label{eq:ineq_observer_iss}
   \Vert \widehat{w}(t,\cdot)-{w}(t,\cdot)\Vert_\L
   \leq c_0 e^{- \lambda t}
    \Vert \widehat w_0-w_0\Vert_\L+
  c_1\int_{0}^t e^{-\lambda(t-s)}\|d_1(s,\cdot)\|_{L^2}ds  + c_2\int_{0}^t e^{-\lambda(t-s)}|d_2(s)|ds , 
\end{equation}
   for any initial conditions $w_0,\widehat w_0\in  L^2(0,L)$, 
 any  $d_1\in L^2([0,t];L^2(0,L))$, any $d_2\in L^2(0,t)$
and for all $t\geq 0$.
\end{corollary}

\begin{proof}
The proof can be directly inherited from 
Theorem~\ref{theo:ISS-Lyap-obs}
by applying  Gr\"onwall's lemma to inequality \eqref{eq:Uder}.
\end{proof}

As a conclusion of this section, 
we remark that the in view of the exponential stability properties of the observer 
\eqref{eqob}, one can also design 
a local observer for 
the nonlinear KdV model 
\eqref{eq_kdv_nonlin}.
In particular, 
selecting the gain $p$
as in Corollary~\ref{corollary:iss-observer} 
it is possible to show that the
following system
$$
    \left\{
\begin{array}{lll}
\widehat{w}_t+\widehat{w}_{x} +\widehat{w}_{xxx}+
\widehat{w}\widehat{w}_{x}+ p(x)[y(t)
-\widehat{w}_x(t,0)]=0 \,,  &\qquad& (t,x)\in \RR_+\times [0,L] \,, \\
\widehat{w}(t,0)=\widehat{w}(t,L)=0 \,, && t \in \RR_+ \,,\\
\widehat{w}_{x}(t,L)=0 \,, && t \in \RR_+ \,,\\
\widehat{w}(0,x)= \widehat w_0(x) \,, && x \in [0,L] \,,
\end{array}
\right.
\label{eqob_nonlin}
$$
is a locally exponentially ISS observer
for system \eqref{eq_kdv_nonlin}, 
namely inequality
\eqref{eq:ineq_observer_iss}
holds for all
 $w_0, \hat w_0, d_1, d_2$ satisfying
  $$
       \Vert \hat w_0 \Vert_{\L} + \Vert w_0 \Vert_{\L} + 
\lim_{t\to\infty}  \int_{0}^t e^{-\mu(t-s)}\Big(\|d_1(s,\cdot)\|_{L^2} + 
|d_2(s)| \Big)ds 
\,\leq \, \delta
  $$
  for some $\delta$ small enough. 
  The proof is omitted for space reasons
  and can be derived by combining 
  the arguments of the proof 
  of Theorem~\ref{theo:ISS-Lyap-obs}
  with the robustness result established in 
  Corollary~\ref{corollary:ISS-nonlineaire}.
\color{black}

\section{Proof of Theorem~\ref{theo}}
\label{sec:proof} 
Let $T>0$. We prove the statement of the Theorem \ref{theo} for $w_0\in H^3_L(0,L)$, $d_2\in C_0^2([0,T])$ and $d_1\in C^1([0,T],L^2(0,L))$, where we recall that $H^3_L(0,L)$ is defined in \eqref{def-H3L}. \color{black}
Since $H^3_L(0,L)$, $C^2([0,T])$
and $C^1([0,T],L^2(0,L))$
are dense in $L^2(0,L)$,
$L^2(0,T)$  and $L^1([0,T];L^2(0,L))$,
respectively, the result follows for all $w_0\in L^2(0,L)$,  $d_1\in L^1([0,T];L^2(0,L))$
and  $d_2\in L^2(0,T)$, by a standard density argument similar to the one provided in \cite[Lemma 1]{marx2017cone}.
\color{black}

\subsection*{Proof of item (a) of  Theorem~\ref{theo}}
\label{preuveli}
The derivative of the Energy 
\eqref{eq:energy}
gives along solutions 
of the linear KdV model  \eqref{eq_kdv_lin}
a negative term in $w_x(t,0)$.
Moreover, 
Theorem~\ref{theo:ISS-Lyap-obs} shows that 
using such a term 
in the  $w$-dynamics, we are able to 
obtain an ISS-Lyapunov functional $U$.
As a consequence, the main idea of this proof
consists in adding and subtracting the term 
$w_x(t,0)$, multiplied by a coefficient $p(x)$,
in the $w$ dynamics: one term is used to obtain 
the negativity in the $L^2$ norm of the full space 
as in \eqref{eq:Uder}, while the other is treated
as a distributed disturbance $d_1$
and compensated by the negativity of the Energy. 

With the previous points in mind, fix $\lambda=1$ and consider the functions $p$  and $U$ given by Theorem~\ref{theo:ISS-Lyap-obs}. 
Set $\bar p := \Vert p \Vert_\L^2$.
Note that   $\bar p\neq0$ because $p$ is a non-zero function.
We define the
operator $\Pi$ and the function $W$ as follows 
\begin{equation}
    \label{eq:Wdef}
    W(w) := \dfrac{1}{2\bar p \varrho_1} U(w) = \Vert \Pi(w)\Vert_{\L}^2\,, 
    \qquad
    \Pi(w) :=  \dfrac{1}{\sqrt{2\bar p \varrho_1}} \overline \Pi^{-1}(w)\,,
\end{equation}
for all $w\in \L(0,L)$,
where the operator $\overline\Pi$ and the parameter $\varrho_1$ are given  
by Theorem~\ref{theo:ISS-Lyap-obs}.
We show that the statement of the theorem holds
and in particular that the inequalities 
 \eqref{eq:Lyap-sandwich}, 
\eqref{eq:Lyap-ISS} 
are satisfied.
First, in view of 
\eqref{eq:Usandwich}, we obtain 
$$
\dfrac{\underline c}{2\bar p  \varrho_1} \Vert w\Vert_{\L} \leq W(w) \leq 
\dfrac{\bar c}{2\bar p \varrho_1} \Vert w\Vert_{\L}
$$
As a consequence, by recalling 
that $E(w) = \Vert w\Vert^2_{\L}$, 
the inequality  \eqref{eq:Lyap-sandwich}
is satisfied for the function 
$V= E+ W$ with 
$\underline\alpha := 1+\frac{\underline c}{2\bar p\varrho_1}$
and 
$\bar \alpha := 1+\frac{\bar c}{2\bar p\varrho_1}$.

Then, in order to show 
the inequality \eqref{eq:Lyap-ISS} 
we compute the derivative of the functional 
$V$ along the trajectories of the system 
\eqref{eq_kdv_lin}.
We first analyze the time derivative
of the energy $E$. Using \eqref{eq:energy-weak}
and adding the effect of the perturbations $d_1,d_2$, we obtain
\begin{align}
  \notag
    \dot E (w) & = - |w_x(0)|^2
    + 2\int^L_0 w(x) d_1(x) dx + |d_2|^2 
    \\
    &\leq 
     - |w_x(0)|^2 + 
      \dfrac{\underline c}{4\bar p  \varrho_1} 
      \Vert w\Vert_\L^2   
     + \dfrac{4\bar p \varrho_1}{\underline c}\Vert d_1\Vert_\L^2   
+ |d_2|^2    \label{eq:Eder}
\end{align}
where the second inequality has been obtaining
by using the Cauchy-Schwarz and Young inequalities, 
and with the parameters $\underline c,\varrho_1$ given by
Theorem \ref{theo:ISS-Lyap-obs}.
Next, we compute the derivative of $W$
along the trajectories of system \eqref{eq_kdv_lin}.
To this end, we first add and subtract the term $p(x)w_x(t,0)$
to the dynamics, obtaining
\begin{equation}
   \left\{
\begin{array}{lll}
w_t+w_{x} +w_{xxx}- p(x)w_x(t,0)=
-p(x)w_x(t,0)+d_1(t,x)\,,  &\qquad & (t,x)\in \RR_+\times [0,L]\,,\\
w(t,0)=w(t,L)=0\,,  && t\in \RR_+ \,,\\
w_{x}(t,L)=d_2(t)\,, && t\in \RR_+\,,\\
w(0,x)=w_0(x) \,,&& x\in [0,L]\,.
\end{array}
\right.
\label{lin}
\end{equation}
Applying the ISS-Lyapunov inequality \eqref{eq:Uder} along solutions to \eqref{lin}, the derivative of $U$ yields
\begin{align}
    \dot U(w) \leq & - U(w) + \varrho_1 \Vert d_1 -pw_x(0)\Vert_{\L}^2 +  \varrho_2|d_2|^2
    \notag \\
     \leq &- \underline c \Vert w\Vert_\L^2 +2 \varrho_1 \Vert d_1 \Vert_{\L}^2 + 2 \varrho_1 \Vert p\, w_x(0)\Vert_{\L}^2 + \varrho_2 |d_2|^2 \,,
     \label{eq:Uder2}
\end{align}
where in the second inequality we used again the
inequality
\eqref{eq:Usandwich}.
Finally, we can compute the derivative 
of the function $V= E+ W$, with
$W$ defined in \eqref{eq:Wdef},
by combining 
\eqref{eq:Uder2} and  \eqref{eq:Eder}
and using the identity 
\begin{equation}
\label{eq:identity}
- |w_x(0)|^2
+\tfrac{1}{\bar p}\Vert p\, w_x(0)\Vert_{\L}^2 = 0 \,.
\end{equation}
Simple computations (omitted for space 
reason) give the inequality \eqref{eq:Lyap-ISS}
with the choice 
$\alpha  :=  
\frac{\underline c}{4 \bar p  \varrho_1}$, 
$\sigma_1 =\frac{4\bar p  \varrho_1}{\underline c} +\frac{1}{\bar p}$, 
$\sigma_2:=1+\frac{ \varrho_2}{2\bar p  \varrho_1}$. 
This concludes the proof of item (a) of Theorem~\ref{theo}.

\subsection*{Proof of the item (b) of Theorem 2.5}
Consider again the function $V= E+W$
with defined in 
\eqref{eq:Wdef}. 
The derivative of the energy 
\eqref{eq:energy}
along the trajectories of the 
nonlinear system \eqref{eq_kdv_nonlin} 
is computed
as in \eqref{eq:Eder} because the contribution of the nonlinear $ww_x$ is zero.
Next, we compute the time derivative of $W$.
However, due to the presence of the nonlinear term 
$ww_x$ we cannot apply off-the-shelf
the inequality \eqref{eq:Uder}
by including such a term 
in the disturbance $d_1$: it would not be 
bounded with the right norm.
As a consequence, unfortunately, we 
need to revisit and adapt some steps
of the proof of Theorem~\ref{theo:ISS-Lyap-obs}
and in particular we need to compute
the change of coordinates defined in 
\eqref{eq:change_of_coord_gamma},
\eqref{inverse}.
Recalling that we selected $\lambda=1$, the $\gamma$-dynamics reads
\begin{equation}
    \left\{
\begin{array}{lll}
\gamma_t+\gamma_{x} +\gamma_{xxx}+ \gamma =
-\bar\Pi^{-1}(p)w_x(t,0)+\overbar\Pi^{-1}(d_1) -\overbar\Pi^{-1}(ww_x)+Q_z(x,L)d_2, && (t,x)\in \Omega\\
\gamma(t,0)=\gamma(t,L)=0\,, && t\in \RR_+\\
\gamma_x(t,L)=d_2(t) \,, && t\in \RR_+\\
\gamma(0,x)=\gamma_0(x) \,, && x\in [0,L]\,.
\end{array}
\right.
\label{eqgammaa}
\end{equation}
where $Q$ is defined in \eqref{inverse}. 
With respect to system \eqref{eqgama_d}
we have two extra terms to analyse, 
that are the terms $\overbar\Pi^{-1}(p)w_x(0)$ and $\overbar\Pi^{-1}(ww_x)$.
As a consequence, we consider again the Lyapunov functional $U(w):=\Vert \gamma\Vert^2_\L$ as in \eqref{eq:Udef}, 
and we follow similar computations to 
those developed from   \eqref{eq:tempUder}
to \eqref{eq:tempUder2}.
Also, as in the proof of item (a), 
we consider as a a full disturbance 
the term $d_1 -pw_x(0)$, see inequality 
\eqref{eq:Uder2}.
In particular, 
the derivative of $U$  along the trajectories of system \eqref{eqgammaa} satisfies, for all   $w\in \L(0,L)$
% \begin{align}\notag
%         \dot U(w) \leq& \displaystyle-2\Vert \gamma\Vert_\L^2
% +d_2^2 -\gamma_x(0)^2+2\left|\int_0^L \overbar \Pi^{-1}(d_1)\gamma(x)dx\right|  +2\left|\int_0^L f(ww_x)\gamma dx\right|
% \\ \label{eq:Uder34}
% &+ 2\left|d_2\int_0^L Q_z(x,L)\gamma(x) dx\right| + 2\left|w_x(0)\int_0^L  \overbar\Pi^{-1}(p)\gamma dx\right|
% \end{align}
$$
\dot U(w)  \leq -\Vert \gamma\Vert_\L^2 + 
\varrho_1 \Vert d_1 -pw_x(0)\Vert_{\L}^2
+ \varrho_2|d_2|^2  
  +2\left|\int_0^L f(ww_x)\gamma dx\right|
$$
where the function $f$ is defined as 
$f(ww_x)(x):=\bar \Pi^{-1}(ww_x)(x)= w(x)w_x(x)+\int^L_0 Q(x,z)w(z)w_x(z)dz.$ By using the same argument as in \cite[Proof of Theorem 1.2, page 1111-1113]{lucoron}, we can show the existence of  positive constant
$\bar f$ that depends only on the function $Q$,
such that
$$
    2\left|\int_0^L f(ww_x)(x)\gamma(x) dx\right|\leq \bar f\Vert \gamma\Vert^3_\L
    \qquad\forall\,w\in \L(0,L).
    % \quad 2\left|\int_0^L \overbar \Pi^{-1}(d_1)\gamma dx\right|\leq f_1 (\Vert d_1\Vert^2_\L + \Vert \gamma \Vert^2_\L) 
$$
%  Using first Cauchy-Schwarz's inequality and then Young's inequality,
%  from \eqref{eq:Uder34} and \eqref{aw},
As a consequence, combining  the previous inequalities
and following the same computations in 
\eqref{eq:Uder2},
 we  obtain, for all $w\in \L(0,L)$
% \begin{align}\notag
%  \dot U(w)  \leq&\displaystyle-\left(2\lambda - f_0 - f_1\Vert \gamma\Vert_\L\right)\Vert \gamma\Vert_\L^2
% +d_2^2 +2\Vert \overbar{\Pi}^{-1} (p)w_x(0)\Vert_{\L}\Vert \gamma\Vert_\L + 2\Vert Q_z(\cdot,L)d_2\Vert_{\L}\Vert \gamma\Vert_\L + f_1 \Vert d_1\Vert^2_\L 
% \\[.5em]
% \notag\leq&\displaystyle
% -\left(\frac{3}{2}\lambda-f_0-f_1\Vert \gamma\Vert_\L\right)\Vert \gamma\Vert_\L^2
% +\left(1+\frac{4}{\lambda}\Vert Q_z(\cdot,L)\Vert_{\L}^2\right)d_2^2+\frac{4}{\lambda}\Vert \bar \Pi^{-1}(p)w_x(0)\Vert_{\L}^2 + f_1 \Vert d_1\Vert_\L^2\,.
% \end{align}
$$
   \dot U(w) \leq   -\left(1- \bar f \Vert\gamma\Vert_{\L} \right)   \Vert \gamma\Vert_\L^2  + 2 \varrho_1 \Vert d_1 \Vert_{\L}^2 + 2 \varrho_1 \Vert p\, w_x(0)\Vert_{\L}^2 + \varrho_2 |d_2|^2\,.
$$
Therefore, 
using the inequality \eqref{eq:Usandwich}, 
we obtain
$$
 \dot U(w) \leq
-\tfrac{\underline c}2  \Vert w\Vert_\L^2  +2 \varrho_1 \Vert d_1 \Vert_{\L}^2 + 2 \varrho_1 \Vert p\, w_x(0)\Vert_{\L}^2 + \varrho_2 |d_2|^2
$$
for all 
 $w$ satisfying $\Vert  w\Vert_{\L} \leq  \bar\delta$, 
 with $\bar\delta =  (2 \sqrt{\underline c} \bar f)^{-1} $.
 Using the definition of the function $W$
 in \eqref{eq:Wdef} and following the same 
 steps of item (a), 
 we obtain the inequality 
 inequality \eqref{eq:Lyap-ISS}
with the choice 
$\alpha  :=  
\frac{\underline c}{8 \bar p  \varrho_1}$, 
$\sigma_1 =\frac{4\bar p  \varrho_1}{\underline c} +\frac{1}{\bar p}$, 
$\sigma_2:=1+\frac{ \varrho_2}{2\bar p  \varrho_1}$, 
and $\delta = \tfrac13\bar \delta$.
 \color{black}

\bigskip

\color{black}
\section{Adding an integral action}
\label{Add}

In this section we consider the regulation problem of a KdV
equation in which the disturbance $d_2$ is considered as a 
control input acting at the boundary condition, 
and the output $y(t)= w_x(t,0)$ has to be regulated 
at a certain desired constant reference $r$
in presence of unknown distributed constant disturbances $d_1$.
We aim at showing that such a problem 
can be solved by means of an integral action and 
an output feedback control law. 
The proposed design 
 is based on the forwarding method (see e.g., \cite{terrand2019adding} or \cite{marx2021forwarding}). 
Note that in Section~\ref{sec_reg_linear}, we  focus on the linearized version of the KdV model \eqref{eq_kdv_lin}. Then, in Section~\ref{sec_reg_nonlinear}, we will show a local result for the nonlinear system
\eqref{eq_kdv_nonlin}. 

% The rest of the section is organized as follows: we first focus on the well-posedness and on the regulation of the system. We will prove that there exists an equilibrium point to the considered system. We second prove that the proposed controller allows to make converge the trajectories of the linear system towards the equilibrium point of the nonlinear system and also to reject the perturbation.

\subsection{Regulation of linear {KdV} equation by means of the forwarding method}
\label{sec_reg_linear}

Consider the following system 
\begin{equation}\label{sys_reg}
\left\{
\begin{array}{lll}
w_t+w_{x} +w_{xxx}=d(x)\,,  &\qquad & (t,x)\in \RR_+\times [0,L]\,,\\
w(t,0)= w(t,L) = 0 \,, & & t\in \RR_+ \\
w_{x}(t,L)=u(t)\,, & & t\in \RR_+\\
w(0,x)=w_0(x)\,, & & x\in [0,L]\\
y(t) = w_x(t,0) \,, & & t\in \RR_+
\end{array}
\right.
\end{equation}where $d\in \L(0,L)$ is a constant perturbation, $u\in \mathbb{R}$ is the control input, 
 and $y\in \mathbb{R}$ is  the output to be regulated 
at a certain desired constant reference 
$r$. We define the regulated output error $e= y-r$ and we
defined our regulation objective as
\begin{equation}
    \label{eq:regulation_objective}
    \lim_{t\to\infty} e(t) =  \lim_{t\to\infty}  y(t) - r = 0\,.
\end{equation}\color{black}
To this end, we follow the standard set-up of output regulation 
\cite{astolfi2017integral,terrand2019adding}
and we extend system 
\eqref{sys_reg} with an integral action 
processing the desired error to be regulated. In other words, 
we consider a dynamical feedback law of the form 
\begin{equation}
\label{control}
  \dot \eta =  y - r \,, \qquad   u=k \eta\,,
\end{equation}
where $\eta\in \RR$ is the state of the controller and $k$ is  a positive constant to be selected small enough, as shown later.  
The closed-loop system 
 \eqref{sys_reg}, \eqref{control} can be seen as an augmented system, i.e. a PDE system (whose state is $w$) coupled with an ODE (whose state is $\eta$), which reads
\begin{equation}\label{sys_reg_closed_loop}
\left\{
\begin{array}{lll}
w_t+w_{x} +w_{xxx}=d(x) \,,  &\qquad & (t,x)\in \RR_+\times [0,L]\,,\\
w(t,0)= w(t,L) = 0 \,,  &\qquad & t\in \RR_+\,,\\
w_{x}(t,L)=k\eta(t)\,,  &\qquad & t\in \RR_+\,,\\
\dot \eta(t) = w_x(t,0) -r \,,  &\qquad & t\in \RR_+\,,\\
w(0,x)=w_0(x),\; \eta(0)=\eta_0 \,,  &\qquad & x\in [0,L]\,.
\end{array}
\right.
\end{equation}
We define the space $X:=\mathbb{R}\times\L(0,L)$, that is the state space of \eqref{sys_reg_closed_loop}. It is a Hilbert space as the Cartesian product of two Hilbert spaces.
In the rest of the section, 
we will show the following properties for the closed-loop system 
\eqref{sys_reg_closed_loop}: it is well posed, it admits a unique equilibrium which is exponentially stable, and the regulation objective
\eqref{eq:regulation_objective} is achieved when considering 
sufficiently regular solutions.

To this end, we introduce now the following two linear operators
$\cS$ and $\cA$
that will be used in the rest of the section.
In particular, 
we denote with 
$\cS$  the operator  associated with the  linear KdV equation
\eqref{eq_kdv_lin}.
The operator $\cS$ and its domain $D(\cS)\subset \L(0,L)$ are defined as
\begin{equation}\label{operator_S}
    \cS w=-w' -w''', 
    \qquad D(\cS):= \lbrace w\in H^3(0,L):\: w(0)=w(L)=w'(L)=0\rbrace.
\end{equation}
Then, we define the operator $\cA$ in order to describe
the closed-loop system \eqref{sys_reg_closed_loop} in the 
following
abstract form  
\begin{equation}\label{operator_A}
\dfrac{d}{d t} \zeta = \cA \zeta+\Gamma\,,\quad
\zeta(0) =\tilde\zeta_0\,,
\quad
\zeta := \begin{pmatrix}
  \eta \\ w
\end{pmatrix}, \quad
\cA(\eta,w) := \begin{bmatrix}
w'(0) \\ -w'-w'''
\end{bmatrix},
\quad
\Gamma:=\begin{bmatrix}
-r \\ d
\end{bmatrix},
\end{equation}
with the domain of $\cA$  defined as 
$D(\cA):= \lbrace (\eta,w)\in \mathbb{R}\times H^3(0,L)\mid w(0)=w(L)=0, w'(L)=k \eta \rbrace \subset X$.
We start by proving the existence and uniqueness of an equilibrium 
for system \eqref{sys_reg_closed_loop}
in the following lemma.
\color{black}

\medskip

\begin{lemma}\label{lemma_equilibrium}
For any $k\neq0$ and $(d, r)\in \L(0,L)\times\RR$ there exist a  unique equilibrium state $(\eta_\infty,w_\infty)\in X$ to
system \eqref{sys_reg_closed_loop}. \\
\label{equi}
\end{lemma}

\begin{proof}
Consider the following boundary value problem 
\begin{equation*}%\label{sys_equilibre}
\left\{
\begin{array}{lll}
w_\infty'(x) +w_\infty'''(x)=d(x)\,, &\qquad & x\in [0,L]\,, \\
w_\infty(0)= w_\infty(L) = 0\,, \\
w_\infty'(0)= r\,,
\end{array}
\right.
\end{equation*}
which represents the nonzero equilibrium state of \eqref{sys_reg_closed_loop}, 
together with $\eta_\infty = \frac{w_\infty'(L)}{k}$.
Consider the smooth function $\phi(x)= \frac{rx(L-x)}{L}$. It  satisfies the boundary conditions $\phi(0)=\phi(L)=0$ and 
$\phi'(0)=r$. We set $\psi=w_\infty-\phi$. Then $\psi$ satisfies the following system
\begin{equation*}%\label{sys_equilibre_2}
\left\{
\begin{array}{lll}
\psi'(x) +\psi'''(x)=j(x) \,, &\qquad & x\in [0,L]\,, \\
\psi(0)= \psi(L) = 0 \,, &\\
\psi'(0)= 0 \,,
\end{array}
\right.
\end{equation*} 
where $j(x)=d(x)-\phi'(x)$. 
This system can be written in the operator form 
as $\cS^*\psi=j$, 
where $\cS^*$, 
is the adjoint operator of $\cS$ defined 
in \eqref{operator_S}. In particular, $\cS^*$, 
is defined as 
\color{black}$\cS^*\psi=\psi'''+\psi'$  with domain $D(\cS^*):= \lbrace w\in H^3(0,L):\: w(0) = w(L)=w'(0)=0\rbrace$. Following \cite[Lemma 4]{marx2015stabilization}, we can prove that the canonical embedding from $D(\cS^*)$, equipped with the graph norm, into $L^2(0,L)$, is compact. Then, according to \cite[Proposition 4.24]{cheverry2019handbook}, $\cS^*$ is an operator with compact resolvent. 
This implies that its spectrum
consists only of eigenvalues.
Moreover, $0$ is not an eigenvalue of $\cS^*$. Hence, there exists a unique solution $\psi_\infty$ to the equation $\cS^*\psi=j$. The equilibrium 
$(\eta_\infty,w_\infty)$ can then be computed as
$ w_\infty(x)=\psi_\infty + \phi(x)$ for all $x\in [0,L]$
and $\eta_\infty = \frac{w_\infty'(L)}{k}$, with 
$\phi$ being the function defined at the beginning
of the proof.  
 \end{proof}
 
Next, we show the following well-posedness result for the
closed-loop system \eqref{sys_reg_closed_loop}.
In the proof, we will also introduce 
a strict Lyapunov functional for the closed-loop system 
\eqref{sys_reg_closed_loop}. Such a Lyapunov functional
is obtained via the forwarding methodology 
similarly 
to \cite{terrand2019adding,marx2021forwarding}
and it is based on the ISS-Lyapunov 
established in Theorem~\ref{theo}.

\medskip

\begin{lemma}\label{theo:well-posedness}
Let $L\notin \mathcal{N}$. There exist $k_0^\star>0$ such that for any $k\in (0,k_0^\star)$, for any $(d, r)\in \L(0,L)\times\RR$ and for any initial condition $(\eta_0,w_0)\in X$ (resp. $D(\cA)$), there exists a unique weak solution $(\eta,w)\in C^0(\RR_+;X)$ (resp. strong solution in $C^1(\RR_+;X)\cap C^0(\RR_+;D(\cA))$) to system \eqref{sys_reg_closed_loop}.
\end{lemma}

\medskip
\begin{proof}
Given $(d, r)\in \L(0,L)\times\RR$
let $(\eta_\infty, w_\infty)$
the corresponding equilibrium to \eqref{sys_reg_closed_loop}
computed according to Lemma~\ref{lemma_equilibrium}.
Consider the following change of coordinates
\begin{equation}\label{eq:changeofcoor_weta}
(w,\eta)\mapsto (\widetilde w, \tilde \eta) :=
(w-w_\infty, \eta-\eta_\infty).
\end{equation}
The $(\widetilde w, \tilde \eta)$-dynamics 
is given by
\begin{equation}\label{sys_reg_closed_loop2}
\left\{
\begin{array}{lll}
\widetilde w_t+ \widetilde w_{x} +\widetilde w_{xxx}=0\,, &\qquad & (t,x)\in \RR_+\times [0,L]\,,\\
\widetilde w(t,0)=\widetilde w(t,L) = 0\,, && t\in \RR_+\,,\\
\widetilde w_{x}(t,L)=k\tilde \eta(t)\,, && t\in \RR_+\,, \\
\dot{\tilde \eta} (t) =\widetilde w_x(t,0) \,, && t\in \RR_+\,,\\
\widetilde w(0,x)=\widetilde w_0(x),\;\tilde \eta(0)=\tilde \eta_0\,, && x\in [0,L]\,,
\end{array}
\right.
\end{equation}
where $\widetilde w_0(x)=w_0(x)-w_\infty(x)$ and $\tilde \eta_0=\eta_0-\eta_\infty$. 
System \eqref{sys_reg_closed_loop2} can be rewritten, 
in the operator form, as 
$$
\dfrac{d}{dt}\tilde\zeta = \cA \tilde \zeta , 
\qquad \tilde \zeta = \zeta_0, 
\qquad
\tilde \zeta := 
\begin{pmatrix}
\tilde \eta\\ \widetilde w
\end{pmatrix}
$$
with $\cA$ and its domain $D(\cA)$ defined as in \eqref{operator_A}.
As a consequence,
systems \eqref{sys_reg_closed_loop} and \eqref{sys_reg_closed_loop2} are equivalent.  Then, if one proves that the operator $\cA$ defined in \eqref{operator_A} is a $m$-dissipative operator on $(X, \|\cdot\|_X)$, one can apply the result provided by \cite[Theorem 3.1]{brezis(1973)}, and conclude that the statement of Lemma~\ref{theo:well-posedness} holds. 
 For that, we look for an equivalent norm and a related scalar product coming from a Lyapunov functional. We will prove then the dissipativity with respect to such a  scalar product. This Lyapunov functional is built following the forwarding approach  (see e.g \cite{terrand2019adding}). 
 To simplify the notation, in the rest of this proof, we will 
 write  $(\eta,w)$ instead of $( \tilde \eta,\widetilde w)$.

Now, by recalling the definition of the operator 
 $\cS$ in given in \eqref{operator_S}, \color{black}
 we define the operator $\cM:L^2(0,1)\rightarrow \RR$ as
solution to the following Sylvester equation
\begin{equation}
\label{MM-Sylvester}    
\cM \cS w = \cC w \,,\qquad\forall w \in D(\cS)\,,
\end{equation} 
where $\cC: f\in H^1_0(0,L)\mapsto f'(0)\in \mathbb{R}$.
Since the strongly continuous semigroup generated by the operator $\cS$ is exponentially stable, the Sylvester equation \eqref{MM-Sylvester} admits a unique solution, see \cite[Lemma 22]{phong1991operator}.  Moreover, since $\cM$ is a linear form, according to Riesz representation theorem \cite[Theorem 4.11]{brezis2010functional}, the operator $\cM$ is uniquely defined as $\cM w=\int^L_0M(x)w(x)dx$.  In order to obtain an explicit solution, we  write 
 equation \eqref{MM-Sylvester} in the explicit form 
$$
 w'(0)  = 
-\int_0^L M(x)[w'(x) +w'''(x)]dx \qquad\forall w \in D(\cS).
$$
Using integration by parts we obtain 
$$
 w'(0)  = 
\int_0^L w(x)[M'(x) +M'''(x)]dx +M(0)w''(0)-M(L)w''(L)-M'(0)w'(0)\,, 
$$ 
for all $w \in D(\cS)$.
From the latter equation, we obtain the following boundary value problem
\begin{equation}\label{mm}
\left\{
\begin{array}{ll}
 M'''+M'=0\,, &\\
 M(0)=M(L) = 0 \,, &\\
M'(0)=-1 \,.
\end{array}
\right.
\end{equation}
It can be verified  that the function
\begin{equation}\label{M}
    M:x\in\mathbb R\mapsto\dfrac{-2\sin({\frac{x}{2}})\sin({\frac{L-x}{2}})}{\sin({\frac{L}{2}})}
\end{equation}
is a solution to \eqref{mm}.
Computations are omitted for space reasons.
Moreover, it is the unique solution to \eqref{mm} and the operator $\cM$ defined above is the unique solution to
the Sylvester equation \eqref{MM-Sylvester}.
Then, the operator $\cM:\L(0,L)\to \mathbb R$ can be expressed as
$\cM\varphi=\int^L_0M(x)\varphi(x)dx$.

With the operator $\cM$ so defined, 
consider the candidate Lyapunov functional $\VV: X \to \RR$ defined as
\begin{equation}
    \label{lyapunov}
    \VV(\eta,w)=V(w)+(\eta-\cM w)^2\,, 
\end{equation} 
where   $V$  is the Lyapunov functional given by Theorem \ref{theo}.
By construction,  the Lyapunov functional $\VV$ is equivalent to the standard norm on the space $X$, and in particular,  there exist positive constants $\underline \nu,\bar \nu$   such that
the following holds
\begin{equation}
\label{ve}
  \underline \nu\Vert (\eta,w)\Vert_X^2\leq \VV(\eta,w)\leq \bar\nu\Vert (\eta,w)\Vert_X^2\,,
    \qquad 
    \forall (\eta,w)\in X\,.
    \qquad 
\end{equation} 
To show this fact, note that, 
following similar arguments used in
the proof of Proposition 4 of 
\cite{terrand2019adding},
for any $\rho\in ]0,1[$ 
we have
$$
 \rho \left(\frac{1}{2}\eta^2 - \Vert M\Vert_{\L}^2\Vert w\Vert_{\L}^2\right)
\leq 
(\eta-\cM w)^2\leq 2(\eta^2+ \Vert M\Vert_{\L}^2\Vert w\Vert_{\L}^2)
$$
for all $(\eta,w)\in X$. \color{black}
Furthermore,
according to Theorem \ref{theo}, we know that $V$ satisfies the inequality \eqref{eq:Lyap-sandwich}. 
Then we have
$$
   \rho \left(\frac{1}{2}\eta^2 - \Vert M\Vert_{\L}^2\Vert w\Vert_{\L}^2\right)+ \underline\alpha\Vert w\Vert_\L^2\leq \VV(w)\leq 2\left(\eta^2+ \Vert M\Vert_{\L}^2\Vert w\Vert_{\L}^2\right)+ \bar\alpha\Vert w\Vert_\L^2\,.
$$  
Therefore, by selecting $\rho$ sufficiently small,
inequality \eqref{ve} holds for some $\bar \nu > \underline \nu>0$.
 By recalling that the function $V$ established 
 in Theorem~\ref{theo} is of the form $V = E + W$, where 
 $E$ and $W$ are quadratic forms of the $L^2$ norm of $w$,
 from the Lyapunov functional $\VV$ defined in \eqref{lyapunov}, 
 we can also   deduce a scalar product, that we define as follows
\begin{equation}
\label{lyapunov-scalar-product}
\left\langle \begin{bmatrix}
\eta_1 & w_1
\end{bmatrix}^\top,\begin{bmatrix}
\eta_2 & w_2
\end{bmatrix}^\top \right\rangle_{\VV} := \big(\eta_1-\cM w_1\big)\big(\eta_2-\cM w_2\big) + \langle w_1,w_2\rangle_\L+\langle \Pi w_1,\Pi w_2\rangle_\L\,,
\end{equation} 
with  $\Pi$ being the linear operator given by  Theorem~\ref{theo}.
\color{black}
It is equivalent to the usual scalar product in $X$.

Now, we are in position to prove that $\cA$ is $m-$dissipative
according to \cite{pazy}.  For this,  we need to show
that  $\cA$ is dissipative and maximal.
We begin with showing the dissipative properties. To this end,  we use the scalar product given in \eqref{lyapunov-scalar-product}. 
By using the definition of $\cA$ given in \eqref{operator_A}, 
we obtain, for all $\zeta\in D(\cA)$, 
\begin{align}
\label{dissipative}
\notag
\langle \cA\zeta,\zeta\rangle_{\VV} =&\big(w'(0)+\cM(w'''+w')\big)\big(\eta-\cM w\big)-\langle w'+w''',w\rangle_\L-\langle \Pi(w'''+w'),\Pi w\rangle_\L\\
=&\Big( w'(0)+\int_0^L M(x)[ w'(x) + w'''(x)]dx\Big)\Big( \eta -\cM w\Big)-\langle w'+w''',w\rangle_\L 
%\\ & 
-\langle \Pi(w'''+w'),\Pi w\rangle_\L \,.
\end{align} 
For the first term, it can be shown, after some integrations by parts,
that \color{black}
\begin{equation}
\label{ma}
  \int_0^L M(x)[ w'(x) + w'''(x)]dx=-k \eta -   w'(0)
 \end{equation}
 for all $\zeta\in D(\cA)$.
 Then,
for the second term, we 
recall the ISS properties of the function $V$ stated in 
Theorem~\ref{theo}. In particular, applying
the inequality \eqref{eq:Lyap-ISS}
to the system \eqref{sys_reg_closed_loop}, in which
$d$ is the distributed disturbance (thus having the role of $d_1$) and
$k\eta$ is seen as a disturbance acting at the boundary condition (thus having the role of $d_2$), we obtain \color{black} 
 \begin{equation}
 \label{maa}
     -2\langle w'+w''',w\rangle_\L-2\langle \Pi(w'''+w'),\Pi w\rangle_\L\leq -\alpha \Vert  w\Vert_\L^2  + \sigma_2 k^2 \eta^2  \,.
 \end{equation}
 for all $\zeta\in D(\cA)$.
 Hence, \color{black}combining inequalities \eqref{dissipative}
 with \eqref{ma} and \eqref{maa}, we obtain
 \begin{align*}
\label{dissipative1}
\langle \cA\zeta,\zeta\rangle_{\VV} \leq & -k\eta\big(\eta-\cM w\big)
-\frac{\alpha}{2} \Vert  w\Vert_\L^2  + \frac{\sigma_2}{2} k^2 \eta^2 
\\
& \leq - k\left(1 - \left(\frac{\sigma_2}{2} + \dfrac{\Vert M\Vert_{\L}}{4\alpha} \right)k  \right) \eta^2 - \dfrac\alpha4 \Vert w\Vert^2_{\L} 
\end{align*} 
for all $\zeta\in D(\cA)$, 
where the second inequality has been obtained by using
 Young's inequality.
 As a consequence, we can select
$$
    k^\star_0=\left(\frac{\sigma_2}{2} + \dfrac{\Vert M\Vert_{\L}}{4\alpha} \right)^{-1}.
$$
This implies that for any $k\in(0,k_0^\star)$
there exists $\varepsilon>0$ such that 
we have 
\begin{equation}
    \label{eq:A_dissipative}
    \langle \cA\zeta,\zeta\rangle_{\VV}\leq - \varepsilon (|\eta|^2+  \Vert w\Vert^2_{\L} )
\end{equation}
for all $\zeta\in D(\cA)$, which shows that the operator $\cA$ is dissipative.

Now, we want to show that $\cA$ is a maximal operator. \color{black}
According to  L\"umer$-$Phillips theorem \cite[Theorem 4.3]{pazy}, proving that $\cA$ is maximal reduces to show that there exists a positive
$\lambda_0$
such that  for all $\zeta\in X$, there exists $\tilde{\zeta}\in D(\cA)$ such that
$(\lambda_0I_X-\cA) \tilde{\zeta} = \zeta$.
Let $(\eta,w)\in X$. We look for a
$(\tilde{\eta},\tilde{w})\in D(\cA)$ satisfying
\begin{equation}
\label{eq-maximal}
\left\{
\begin{array}{lll}
\tilde{w}''' +\tilde{w}'+\lambda_0\tilde{w}=w\,, &\qquad & x\in [0,L]\,, \\
\tilde{w}(0)=\tilde{w}(L)=0\,, \\
\tilde{w}'(L)=k\tilde{\eta}\,,\\
\lambda_0\tilde{\eta}-\tilde{w}'(0)=\eta\,,
\end{array}
\right.
\end{equation} 
namely
\begin{equation*}
%\label{eq-maximal1}
\left\{
\begin{array}{lll}
\tilde{w}''' +\tilde{w}'+\lambda_0\tilde{w}=w\,, &\qquad& x\in [0,L]\,, \\
\tilde{w}(0)=\tilde{w}(L)=0\,, \\
\tilde{w}'(L)=\frac{k}{\lambda_0}(\eta+\tilde{w}'(0))\,,\\
\lambda_0\tilde{\eta}-\tilde{w}'(0)=\eta\,.
\end{array}
\right.
\end{equation*}
Now, we consider the following boundary  value problem
\begin{equation*}%\label{eq-maximal2}
\left\{
\begin{array}{lll}
\tilde{w}''' +\tilde{w}'+\lambda_0\tilde{w}=w\,, &\qquad& x\in [0,L]\,, \\
\tilde{w}(0)=\tilde{w}(L)=0\,, \\
\tilde{w}'(L)=\frac{k}{\lambda_0}(\eta+\tilde{w}'(0))\,,
\end{array}
\right.
\end{equation*} 
and the smooth function $\tilde\phi(x)= \frac{k\eta x^2(x-L)}{\lambda_0 L^2}$ satisfying the boundary conditions
$$
\tilde\phi(0)=\tilde\phi(L)=\tilde\phi'(0)=0\,,\qquad \tilde\phi'(L)=\frac{k}{\lambda_0}\eta\,.
$$ 
We set $\tilde\psi=\tilde w-\tilde\phi$. Then $\tilde\psi$ satisfies the following boundary value problem
\begin{equation}\label{sys_3}
\left\{
\begin{array}{lll}
\tilde\psi' +\tilde\psi'''+\lambda_0\tilde\psi=\tilde j(x)\,, &&
x\in [0,L]\,,\\
\tilde\psi(0)=\tilde \psi(L) = 0 \,,&\\
\tilde\psi'(L)=\frac{k}{\lambda_0}\tilde \psi'(0)\,,
\end{array}
\right.
\end{equation} where $\tilde j(x)=w(x)-\tilde\phi'(x)-\tilde\phi'''(x)-\lambda_0\tilde\phi$. 
Now, we define the operator $\widehat{\cS}$ and its domain $D(\widehat{\cS})\subset L^2(0,L)$  as
$$
\widehat{\cS}\psi=-\psi' -\psi''', 
\qquad 
D(\widehat{\cS}):= \Big\lbrace\psi \in H^3(0,L): \psi(0)=\psi(L)=0, \psi^\prime(L)=\frac{k}{\lambda_0}\psi^\prime(0)\Big\rbrace \,.
$$
We define also its  adjoint operator $\widehat{\cS}^*$ 
and its domain $D(\widehat{\cS}^*)$ as
$$
\widehat{\cS}^*\psi=\psi'''+\psi', 
\qquad
D(\widehat{\cS}^*):= \Big\lbrace\psi \in H^3(0,L): \psi(0)=\psi(L)=0,\psi^\prime(0)=\frac{k}{\lambda_0}\psi^\prime(L)
\Big\rbrace \,.
$$
Note that $\widehat{\cS}$ and $\widehat{\cS}^*$ are dissipative. 
Indeed, by selecting $\lambda_0>k$, we have 
\begin{equation*}
\begin{array}{rcll} \displaystyle
    \int^L_0\psi\widehat{\cS}\psi dx=\left(\frac{k}{\lambda_0}-1\right)\psi^\prime(0)^2 &<& 0 \,,
\qquad \qquad & \psi\in D(\widehat{cS})\,,
\\[1em] \displaystyle
    \int^L_0\psi\widehat{\cS}^*\psi dx=\left(\frac{k}{\lambda_0}-1\right)\psi^\prime(L)^2&<& 0 \,,
    & \psi\in D(\widehat{\cS}^*)\,.
    \end{array}
\end{equation*} 
Moreover, $\widehat{\cS}$  is closed and $D(\widehat{\cS})$ is dense in $\L(0,L)$. Then, according to \cite[Theorem 4.3 and Corollary 4.4]{pazy} $\widehat{\cS}$ is $m$-dissipative operator.
Finally, since $\widehat{\cS}$ is a $m$-dissipative operator then the system \eqref{sys_3} admits a solution $\tilde\psi$ in $D(\widehat{\cS})$. As a consequence, there exist $(\tilde \eta,\tilde w)\in D(\cA)$ solution of \eqref{eq-maximal}. This proves that $\cA$ is maximal and concludes the proof of Lemma~\ref{theo:well-posedness}.
\end{proof} 

Finally, the next result deals with the exponential stability of equilibrium state  $(\eta_\infty, w_\infty)$ and with the related output regulation
 objective \eqref{eq:regulation_objective}. 

\medskip
\begin{theorem}[Stabilization and regulation]\label{regulation}
Let $L\notin \mathcal{N}$ and consider system \eqref{sys_reg_closed_loop}. 
For any  $k\in (0,k_0^\star)$,
with  $k_0^\star$ given by Lemma~\ref{theo:well-posedness}, 
there exist
$b_0,\nu_0>0$, and for any 
 $(d, r)\in \L(0,L)\times\RR$
 there exists $(\eta_\infty,w_\infty)\in X$, 
 computed according to Lemma~\ref{lemma_equilibrium},
 such that 
any solution to system \eqref{sys_reg_closed_loop} 
with initial condition $(\eta_0,w_0)\in X$
satisfies
\begin{equation}
\label{exp}
    \Vert (\eta(t),w(t,\cdot))-(\eta_\infty,w_\infty)\Vert_X\leq b_0 e^{-\nu_0 t}\Vert (\eta_0,w_0)-(\eta_\infty,w_\infty)\Vert_X.
\end{equation}
for all $t\geq0$.
 Moreover, for any strong solution to \eqref{sys_reg_closed_loop}, 
 and in particular, for any 
 $(\eta_0,w_0)\in D(\cA)$,
 the output $y$ is asymptotically regulated at the reference $r$, namely 
 \eqref{eq:regulation_objective} is satisfied.
\end{theorem}
\medskip

\begin{proof}  
The first part of the proof is proved for 
 any initial condition $(\eta_0, w_0)\in D(\cA)$.
 The result follows for all initial conditions in $X$ by a standard density argument 
(see e.g. \cite[Lemma 1]{marx2017cone}). 
Consider the  equilibrium   $(\eta_\infty, w_\infty)$,
recall the change of coordinates 
defined in  \eqref{eq:changeofcoor_weta} and consider
the error system \eqref{sys_reg_closed_loop2}.
We show  now that the its origin
is  exponentially stable.
 To this end,  consider the Lyapunov functional $\VV$ defined in \eqref{lyapunov}. 
According to the proof of dissipativity of $\cA$ of Lemma~\eqref{theo:well-posedness}, 
for any $k\in (0,k^\star)$
the time derivative of $\VV$ 
along the strong solution to \eqref{sys_reg_closed_loop2} 
satisfies \eqref{eq:A_dissipative}. \color{black}
As a consequence, from \eqref{ve} and  Gr\"onwall's lemma, there exist positive constants $b_0, \nu_0$  such that,  for all $(\eta_0,w_0)\in D(\cA)$ and for all $t\geq0$
\begin{equation}\label{barw}
    \Vert (\tilde \eta(t),\widetilde w(t,\cdot))\Vert_X\leq b_0 e^{-\nu_0 t}\Vert (\tilde \eta_0,\widetilde w_0)\Vert_X.
\end{equation} 
By using the density of  $D(\cA)$ in $X$, 
and the change of coordinates 
 \eqref{eq:changeofcoor_weta}, 
we conclude that \eqref{exp} holds.

Now, we need to show that the regulation objective
\eqref{eq:regulation_objective} is achieved for strong solutions. 
\color{black} For this, note that if  $(\eta_0,w_0)\in D(\cA)$, then $(\tilde \eta_0,\widetilde w_0)\in D(\cA)$. Then $(\tilde \eta,\widetilde w)\in C^1(\RR_+;X)\cap C^0(\RR_+;D(\cA))\big)$. Now, let us introduce the new variables $v, \xi$ defined as follows \begin{equation}
    \label{eq:change_of_coord_vxi}
(\widetilde w, \tilde \eta) \mapsto  (v,\xi) := (\widetilde w_t ,\dot{\tilde \eta}) \,.
\end{equation}
The dynamics of $(v, \xi)$ is given as
\begin{equation}\label{bar}
\left\{
\begin{array}{lll}
 v_t+  v_{x} + v_{xxx}=0\,, &\qquad &
(t,x)\in \RR_+\times [0,L]\,,\\
 v(t,0)= v(t,L) = 0 \,,&&t\in \RR_+\\
 v_{x}(t,L)=k \xi(t)\,,&&t\in \RR_+ \\
\dot{\xi} (t) = v_x(t,0) \,,&&t\in \RR_+\\
 v(0,x)=v_0(x),\xi(0)=\xi_0
\,,&&x\in [0,L]\,.
\end{array}
\right.
\end{equation} 
with  
\begin{equation}
    \label{eq:cond_init_vxi}
v_0(x) = -\widetilde w_0'(x)-\widetilde w_0'''(x), 
\quad x\in [0,L], 
\qquad\xi_0= \widetilde w_0'(0).
\end{equation}
Since $(v(0,\cdot), \xi(0))\in X$, then, according to the Lemma~\ref{theo:well-posedness} and the first statement of Theorem \ref{regulation}, we have  $(v, \xi)\in C^0(\RR_+;X)$ and 
\begin{equation*}
    \Vert (\xi(t), v(t,\cdot))\Vert_X\leq b_0e^{-\nu_0 t}\Vert 
    (\xi(0), v(0,\cdot))\Vert_X \,, \qquad \forall 
  (v_0,\xi_0)\in X\,.
\end{equation*}
By definition of $v$ and $\xi$ and using \eqref{eq:cond_init_vxi}, one can see that, once one considers $(v_0,\xi_0)\in X$, then this implies that $(\eta_0,w_0)\in D(\cA)$. 
Then, using 
the definition of
the change of coordinates
\eqref{eq:change_of_coord_vxi} and \eqref{eq:cond_init_vxi}
we obtain
\begin{equation}\label{barwt}
    \Vert \widetilde w_t(t,\cdot) \Vert_\L \leq
     \Vert (\dot {\tilde\eta}(t), \widetilde w_t(t,\cdot)) 
     \Vert_X  
      \leq
      b_0 e^{-\nu_0 t}\Vert (-\widetilde w_0'(x)-\widetilde w_0'''(x),\widetilde w_0'(0))
     \Vert_X \,,
     \qquad \forall w_0\in D(\cA)\,.
\end{equation} \color{black}
Now, by multiplying the first equation of  \eqref{sys_reg_closed_loop2} by $\widetilde w$ and integrating by parts, we get after some computations
\begin{equation*}
    k^2\tilde \eta(t)^2-\widetilde w_x(t,0)^2=\int^L_0 \widetilde w(t,x)\widetilde w_t(t,x)dx\,.
\end{equation*} 
Using Cauchy-Schwarz's inequality, from \eqref{barw} and \eqref{barwt} we finally obtain 
\begin{equation*}
    \vert\widetilde w_x(t,0)\vert^2\leq\Vert \widetilde w(t,\cdot)\Vert_\L\Vert\widetilde w_t(t,\cdot)\Vert_\L+k^2\vert\tilde \eta(t)\vert^2 \underset{t\to\infty}{\longrightarrow} 0\,,
    \qquad \forall (\eta_0,w_0)\in D(\cA)\,.
\end{equation*} 
From the previous inequality  we obtain
$    \lim_{t\to\infty}\vert \widetilde w_x(t,0)\vert= 
     \lim_{t\to\infty}\vert w_x(t,0)-r\vert
 =    0 $
for all $(\eta_0,w_0)\in D(\cA)$,
and therefore 
\eqref{eq:regulation_objective},
concluding the proof.
\end{proof}

\subsection{Regulation of nonlinear KdV equation by means of the forwarding method}
\label{sec_reg_nonlinear}
In this section, we consider the 
regulation problem for a nonlinear KdV equation
\eqref{eq_kdv_nonlin}. In particular, we consider the 
system
\begin{equation}\label{sys_reg_nln}
\left\{
\begin{array}{lll}
w_t+w_{x} +w_{xxx}+ww_x=d(x)\,,& \qquad& (t,x)\in \RR_+\times [0,L]\,, \\
w(t,0)= w(t,L) = 0 \,,&&  t\in \RR_+\,,\\
w_{x}(t,L)=u(t)\,,&&  t\in \RR_+\,, \\
w(0,x)=w_0(x)\,, && x\in [0,L]\,,\\
y(t) = w_x(t,0)  \,, & & t\in \RR_+\,,
\end{array}
\right.
\end{equation} 
where $d\in \L(0,L)$ is a constant  perturbation, $u\in \RR$ is the control input and $y(t)\in \RR$ is the output to be regulated
to a constant reference $r$ as in \eqref{eq:regulation_objective}.
Following the design proposed  in Section~\ref{sec_reg_linear}
for the linear model \eqref{eq_kdv_lin}, we 
consider the  same output-feedback  integral control 
\eqref{control} and 
we compactly write  the  closed-loop system 
\eqref{sys_reg_nln}, \eqref{control} as
\begin{equation}\label{sys_reg_closed_loop_nln}
\left\{
\begin{array}{lll}
w_t+w_{x} +w_{xxx}+ww_x=d(x)\,, &\qquad & (t,x)\in \RR_+\times [0,L]\,,\\
w(t,0)= w(t,L) = 0 \,,&&  t\in \RR_+\,,\\
w_{x}(t,L)=k\eta(t) \,,&&  t\in \RR_+\,,\\
\dot \eta(t) = w_x(t,0) -r \,,&&  t\in \RR_+\,,\\
w(0,x)=w_0(x),\eta(0)=\eta_0  \,, && x\in [0,L]\,.
\end{array}
\right.
\end{equation}
In the following, we will show that 
for sufficiently small perturbations $d$ and references $r$
the closed-loop system \eqref{sys_reg_closed_loop_nln}
is well posed and it admits a unique equilibrium which 
is locally exponentially stable. Furthermore, 
for solutions which are sufficiently regular, 
the regulation objective \eqref{eq:regulation_objective}
is satisfied.
We start by showing the existence and uniqueness
of an equilibrium.

\medskip

\begin{lemma}\label{lemmanonlinear}
There exist 
 $\bar d>0$
and $\bar r>0$  such that, for any $(d, r)\in L^2(0,L)\times \RR$
satisfying  $\Vert d\Vert_\L\leq\bar d$ and  $\vert r\vert\leq \bar r$,  there exists a unique equilibrium state 
$(\eta_\infty,w_\infty)\in X$ to
system \eqref{sys_reg_closed_loop_nln}. 
Furthermore there exists $\overbar w>0$ such that 
for any $\overbar w_0\in (0,\overbar w]$, 
there exists $d_0>0$ and 
$r_0>0$ so that, for any 
$(d, r)\in L^2(0,L)\times \RR$
satisfying 
$\Vert d\Vert_\L\leq d_0$
and 
$\vert r\vert\leq r_0$ then
$\Vert w_{\infty}\Vert_{H^3}\leq \overbar w_0$.
\end{lemma}

\medskip

\begin{proof}
Consider the following boundary value problem 
\begin{equation}\label{sys_equilibre_nln0}
\left\{
\begin{array}{lll}
w_\infty'(x) +w_\infty'''(x)+w_\infty(x)w_\infty'(x)=d(x)\,, &\qquad& x\in [0,L]\,,\\
w_\infty(0)= w_\infty(L) = 0\,, &\\
w_\infty'(0)= r\,,
\end{array}
\right.
\end{equation} 
which represents the nonzero equilibrium state of \eqref{sys_reg_closed_loop}, with 
$\eta_\infty=\frac{w_\infty'(L)}{k}$.
We prove that there exists a solution to  system 
\eqref{sys_equilibre_nln0}
by following a fixed-point strategy. We set 
$$
H^3_r(0,L):=\Big\lbrace w\in H^3(0,L): w(0)=w(L)=0, 
w^\prime(0)=r
\Big\rbrace\,,
$$ 
and we introduce the operator $\cT_0:H^3_r(0,L)\to H^3_r(0,L)$ defined by $\cT_0(w)= \varphi$ where $\varphi$ is the solution to
\begin{equation}\label{sys_equilibre_nln1}
\left\{
\begin{array}{lll}
 \varphi'(x) + \varphi'''(x)=d(x)-w_\infty(x)w_\infty'(x)\,, &\qquad& x\in [0,L]\,,\\
 \varphi(0)=  \varphi(L) = 0 \,,&\\
 \varphi'(0)= r\,,
\end{array}
\right.
\end{equation}
Note that the function $\Vert \cdot\Vert_{H^3_r}$ : $w\in H^3(0,L) \mapsto  \Vert w'+w'''\Vert_\L \in \mathbb{R}_+$ is a semi-norm on the space $ H^3(0,L)$. Furthermore,  $H^3_r(0,L)\subset H^1_0(0,L)$. Then, according to the Poincar\'e's inequality, the semi-norm $\Vert \cdot\Vert_{H^3_r}$ is a norm on the space $H^3_r(0,L)$ which is equivalent to the standard norm induced by $ H^3(0,L)$. In other words, there exists a positive constant $\kappa$  such that 
\begin{equation}\label{equivalenceseminorme}
     \Vert w\Vert_{ H^3_r}\leq \Vert w\Vert_{ H^3(0,L)} \leq \kappa \Vert w\Vert_{ H^3_r} \,,\qquad \forall w\in H^3_r(0,L)\,.
\end{equation}
Now, we have, 
\begin{align*}
\notag
    \Vert \cT_0(w)\Vert_{ H^3_r}=&\Vert d-ww'\Vert_\L\\\notag
    \leq&\Vert d\Vert_\L+\Vert ww'\Vert_\L\\
    \leq&\Vert d\Vert_\L+
    \Vert w\Vert_{L^\infty}\Vert w'\Vert_\L \,,
\end{align*}
for all $w\in H^3_r(0,L)$. Denoting with the constant $\ell$ the norm of the embedding $H^3(0,L)$ in $L^\infty(0,L)$, according to the Rellich-Kondrachov Theorem (see  \cite[Theorem 9.16]{brezis2010functional}), we have
\begin{align*}
\notag
    \Vert \cT_0(w)\Vert_{ H^3_r}\leq&\Vert d\Vert_\L+\ell\Vert w\Vert_{ H^3(0,L)}\Vert w'\Vert_\L\\
    \leq&\Vert d\Vert_\L+\kappa\ell\Vert w\Vert_{ H^3_r}^2\\ \notag
    \leq&\bar d+\kappa\ell\Vert w\Vert_{ H^3_r}^2 \,,
\end{align*}
for all $w\in H^3_r(0,L)$ and all $d$ satisfying $|d|_{\L}\leq \bar d$.
Moreover, we have for all $w_1,w_2\in H^3_r(0,L) $
\begin{align*}
\label{lipshiz}\notag
    \Vert \cT_0(w_1)- \cT_0(w_2)\Vert_{ H^3_r}&=\Vert w_1w'_1-w_2w'_2\Vert_\L\\\notag
    &\leq\Vert (w_1-w_2)w'_1\Vert_\L+\Vert w_2(w'_1-w'_2)\Vert_\L\\\notag
    &\leq \kappa\ell\Vert w_1-w_2\Vert_{ H^3_r}\Vert w_1\Vert_{ H^3_r}+\kappa\ell \Vert w_1-w_2\Vert_{H^3_r}\Vert w_2\Vert_{H^3_r}\\
    &\leq \kappa\ell\big(\Vert w_2\Vert_{ H^3_r}+\Vert w_2\Vert_{ H^3_r}\big)\Vert w_1-w_2\Vert_{ H^3_r}.
\end{align*}
We consider now the operator $\cT_0$
defined as in \eqref{sys_equilibre_nln1},
restricted on the closed ball 
$$
B_{\overbar w}:=\Big\lbrace w\in H^3_r(0,L): \Vert w\Vert_{H^3_r}\leq \overbar w\Big\rbrace
$$ 
with $\overbar w$ to be chosen later. Then, 
collecting all the previous inequality
we have
\begin{align*}
    \Vert \cT_0(w)\Vert_{H^3_r} & \leq\bar d+ \kappa\ell \overbar w^2 \,,
\\
    \Vert \cT_0(w_1)- \cT_0(w_2)\Vert_{ H^3_r} & \leq 2 \kappa\ell\overbar w\Vert w_1-w_2\Vert_{ H^3_r}\,,
\end{align*}
for all $w,w_1,w_2\in B_{\overbar w}$.
Finally, we  select $\bar d$  and $\overbar w$ such that the following conditions hold
\begin{equation}\label{bar R}
 \bar d<\frac{1}{4 \kappa\ell} \qquad \mbox{ and } \qquad   \frac{1-\sqrt{1-4\bar d \kappa\ell}}{2\kappa\ell}\leq \overbar w<\frac{1}{2\kappa\ell}\,.
\end{equation}
With such a choice, we obtain
$ \Vert \cT_0(w)\Vert_{H^3_r} \leq \overbar w$
for all $w\in B_{\overbar w}$
and 
$    \Vert \cT_0(w_1)- \cT_0(w_2)\Vert_{ H^3_r}  < 
    \Vert w_1-w_2\Vert_{ H^3_r}\,,
$
for all $w_1,w_2\in B_{\overbar w}$.
This shows that the operator $\cT_0$ is an operator 
of contraction.
Applying the Banach fixed point theorem \cite[Theorem 5.7]{brezis2010functional} we deduce that the operator 
$\cT_0$  admits a unique fixed point, and therefore 
that there exists a unique solution $w_\infty\in B_{\overbar w}$ to \eqref{sys_equilibre_nln0}.
Now, given $\overbar w$, we deduce the value of $\bar r$. Indeed, since $w_\infty\in H^3(0,L)$ then  we have $w_\infty'\in H^2(0,L)$. Then, according to the embedding of $ H^2(0,L)$ in $C^1([0,L])$, we have $w_\infty'\in C^1([0,L])$. Therefore, according to \cite[Lemma 1]{zheng2018input} we have 
\begin{equation}
\label{lienw*}
   ( w_\infty'(0))^2 \; \leq \; 
   \frac{2}{L}\Vert w_\infty'\Vert_\L^2+L\Vert w_\infty''\Vert_\L^2
  \; \leq\;
   \left(\frac{2}{L}+L\right)\Vert w_\infty\Vert_{H^3(0,L)}^2\,.
\end{equation} 
 Since $w_\infty\in B_{\overbar w}$, then according to \eqref{equivalenceseminorme} and \eqref{lienw*},
and to the definition of $H^3_r(0,L)$,
we obtain 
\begin{equation}\label{rayon}
    r^2 =  ( w_\infty'(0))^2 \leq \kappa \left(\frac{2}{L}+L\right)\overbar w^2\,.
\end{equation} 
Finally, we can choose $\bar r=\overbar w\sqrt{\kappa \left(\frac{2}{L}+L\right)} $. 
Therefore, according to \eqref{bar R} and \eqref{rayon}, we deduce that for any $\overbar w_0\in (0,\overbar w]$,  there exists $d_0>0$ and 
$\bar r_0>0$ so that, for any $(d, r)\in L^2(0,L)\times \RR$
satisfying $\Vert d\Vert_\L\leq \bar d_0$
and $\vert r\vert\leq \bar r_0$ then $\Vert w_{\infty}\Vert_{H^3}\leq \overbar w_0$. This concludes the proof of Lemma~\ref{lemmanonlinear}. 
\end{proof}

\medskip

Now, given  $(d, r)\in L^2(0,L)\times \RR$
satisfying the 
assumptions of Lemma~\ref{lemmanonlinear}, 
let  $(\eta_\infty,w_\infty)$  be the
corresponding 
equilibrium to  system \eqref{sys_reg_closed_loop_nln}
and consider the following change of coordinates
$$
(w,\eta)\mapsto (\widetilde w, \tilde \eta) :=
(w-w_{\infty}, \eta-\eta_{\infty}).
$$
The $(\widetilde w, \tilde \eta)$-dynamics 
is given by
\begin{equation}\label{sys_reg_closed_loop_non}
\left\{
\begin{array}{lll}
\widetilde w_t+ \widetilde w_{x} +\widetilde w_{xxx}+\widetilde w\widetilde w_{x}=-w_{\infty}'\widetilde w-w_{\infty}\widetilde w_{x}\,, &\qquad & (t,x)\in  \RR_+\times [0,L]\,,\\
\widetilde w(t,0)=\widetilde w(t,L) = 0\,, && t\in \RR_+\,,\\
\widetilde w_{x}(t,L)=k\tilde \eta(t)\,, && t\in \RR_+\,, \\
\dot{\tilde \eta} (t) =\widetilde w_x(t,0) \,, && t\in \RR_+\,,\\
\widetilde w(0,x)=\widetilde w_0(x),\;\tilde \eta(0)=\tilde \eta_0\,, && x\in [0,L]\,,
\end{array}
\right.
\end{equation}
where $\widetilde w_0(x)=w_0(x)-w_{\infty}(x) \in H^3(0,L) $ and $\tilde \eta_0=\eta_0-\eta_{\infty} \in \RR$.
In the new coordinates, the regulation objective \eqref{eq:regulation_objective}
for system \eqref{sys_reg_closed_loop_non}
reads 
\begin{equation}\label{eq:regulation_objective2}
  \lim_{t\to\infty}   e(t) =   \lim_{t\to\infty} \widetilde w_x(t,0)  = 0
\end{equation}
Note that showing the 
well-posedness  of system \eqref{sys_reg_closed_loop_non}
is equivalent to prove the well-posedness  of  system \eqref{sys_reg_closed_loop_nln} in the original coordinates $(w,\eta)$.
As a consequence, in the rest of the section, we will focus  
on the system  \eqref{sys_reg_closed_loop_non} in the new coordinates 
$(\widetilde w, \tilde \eta)$.

\medskip

\begin{lemma}\label{theo:well-posedness_nln}
For any $\overbar w_\infty$, there 
exists $k_1^\star>0$ such that, 
for any $k\in (0,k^\star_1]$, for any 
$w_\infty\in H^3(0,L)$
satisfying $\Vert w_\infty\Vert_{H^3} \leq \overbar w_\infty$, 
and  any  initial condition $(\tilde\eta_0,\widetilde w_0)\in D(\cA)$,
 there exists $\tau>0$  such that the Cauchy problem \eqref{sys_reg_closed_loop_non} is well-posed in the space $C^1(0,\tau)\times \big( C([0,\tau];H^3(0,L))\cap L^2([0,\tau];H^{4}(0,L))\big)$.
\end{lemma}

\medskip

\begin{proof}
 First, by writing the explicit solution of  $\tilde \eta$ along solutions, that is 
 $\tilde \eta(t) = \tilde\eta_0 + \int_0^t \widetilde w_x(s,0)ds$, we rewrite system 
 \eqref{sys_reg_closed_loop_non} as follows
\begin{equation}\label{sys_reg_closed_loop_non1}
\left\{
\begin{array}{lll}
\widetilde w_t+ \widetilde w_{x} +\widetilde w_{xxx}+\widetilde w\widetilde w_{x} + (\widetilde ww_{\infty})_{x} = 0\,, &\qquad & (t,x)\in \RR_+\times [0,L]\,,\\
\widetilde w(t,0)=\widetilde w(t,L) = 0\,, && t\in \RR_+\,,\\
\widetilde w_{x}(t,L)=k \left(\tilde \eta_0+\int ^t_0 \widetilde w_x(s,0)ds\right)\,, && t\in \RR_+\,, \\
\widetilde w(0,x)=\widetilde w_0(x),\, && x\in [0,L]\,.
\end{array}
\right.
\end{equation}
Now, given $(s,\tau)\in \NN\times \RR$, 
we introduce the space $S^s(\tau):=C([0,\tau];H^s(0,L))\cap L^2([0,\tau];H^{s+1}(0,L))$ 
equipped with the norm  defined as
 $$\Vert w\Vert_{S^{s}(\tau)}^2:= \Vert w\Vert_{C([0,\tau];H^s(0,L))}^2
+ \Vert w\Vert_{L^2([0,\tau];H^{s+1}(0,L))}^2+ \Vert w_x\Vert_{C([0,\tau];\L(0,L))}^2\,.
$$
We consider the operator $\cT_1:S^3(\tau)\to S^3(\tau)$ defined by $\cT_1(\widetilde w)=\varphi$ where $\varphi$ is 
the solution of 
\begin{equation}\label{sys_reg_closed_loop_non2}
\left\{
\begin{array}{lll}
\varphi_t+ \varphi_{x} +\varphi_{xxx}+(w_\infty\varphi)_x=-\widetilde w\widetilde w_{x}\,, &\qquad & (t,x)\in  [0,\tau]\times[0,L]\,,\\
\varphi(t,0)=\varphi(t,L) = 0\,, && t\in [0,\tau]\,,\\
\varphi_{x}(t,L)=k \left(\tilde \eta_0+\int ^t_0 \widetilde w_x(s,0)ds\right)\,, && t\in [0,\tau]\,, \\
\varphi(0,x)=\widetilde w_0(x),\, && x\in [0,L]\,
\end{array}
\right.
\end{equation} with $\tau>0$ and $k>0$  to be chosen later.
With the operator $\cT_1$ so defined, we deduce that
if $\widetilde w$ is a fixed point of $\cT_1$
then 
$\widetilde w\in S^3(\tau)$ is a solution of \eqref{sys_reg_closed_loop_non1}. 
To this end we will apply the Banach fixed-point Theorem.
According to \cite[Proposition 5.1]{bona2003nonhomogeneous}, 
for any $\overbar w_\infty$
there exists $C>0$ that
 such that  
\begin{multline*}
    \Vert \cT_1(\widetilde w)\Vert_{S^3(\tau)}^2  \leq  C\Bigg(\Vert \widetilde w_0\Vert_{H^3(0,L)}^2+ k^2\left(\tau\vert \bar\eta_0\vert^2 + \int^\tau_0\left\vert \int^t_0 \widetilde w_x(s,0)ds \right\vert^2 dt +\Vert w_x(\cdot,0)\Vert_{\L(0,\tau)}^2\right)  \\ 
    \hfill
    +\Vert \widetilde w\widetilde w_{x}\Vert_{H^{1}(0,\tau;H^1(0,L))}^2 \Bigg)
\end{multline*}
for all $\widetilde w \in S^3(\tau)$ and any
$\Vert w_\infty\Vert_{H^3}\leq \overbar w_\infty$.
On the other hand, we have 
$$
    \int^\tau_0\left\vert \int^t_0 \widetilde w_x(s,0)ds \right\vert^2 dt
   \;\leq\; \int^\tau_0\int^t_0 \vert \widetilde w_x(s,0)\vert ^2ds \; \leq \; \tau\Vert\widetilde w_x(\cdot,0)\Vert_{\L(0,\tau)}^2.
$$
Since, $\widetilde w\in S^3(\tau)$, then for all $t\in [0,\tau]$, $\widetilde w(t,\cdot)\in H^3(0,L)$, which imply  
$\widetilde w_x(t,\cdot)\in H^2(0,L)$
for all $t\in [0,\tau]$.  Then, according to the embedding of $ H^2(0,L)$ in $C^1([0,L])$, we have $\widetilde w_x(t,\cdot)\in C^1([0,L])$  for all $t\in [0,\tau]$. Therefore, according to \cite[Lemma 1]{zheng2018input}, 
we obtain
$$
     ( \widetilde w_x(t,0))^2\leq \left(\frac{2}{L}+L\right)\Vert \widetilde w(t,\cdot)\Vert_{H^3(0,L)}^2\,.
$$
for all $t\in [0,\tau]$,
which  implies 
\begin{equation}
\label{pose1}
    \Vert\widetilde w_x(\cdot,0)\Vert_{\L(0,\tau)}^2\leq  \left(\frac{2}{L}+L\right)\Vert \widetilde w\Vert_{\L(0,\tau;H^3(0,L))}^2\leq \left(\frac{2}{L}+L\right) \Vert\widetilde w\Vert_{S^3(\tau)}^2.
\end{equation}   
Then, we have
\begin{equation}\label{pose2}
      \int^\tau_0\left\vert \int^t_0 \widetilde w_x(s,0)ds \right\vert^2 dt\leq \tau\left(\frac{2}{L}+L\right)\Vert \widetilde w\Vert_{\L(0,\tau;H^3(0,L))}^2\leq \tau\left(\frac{2}{L}+L\right) \Vert\widetilde w\Vert_{S^3(\tau)}^2.
\end{equation} 
Also, since $\widetilde w\in S^3(\tau)$, then according to \cite[Lemma 3.1]{bona2003nonhomogeneous}, we deduce the existence of $C>0$ such that 
\begin{equation}\label{pose3}
    \Vert \widetilde w\widetilde w_{x}\Vert_{H^{1}(0,\tau;H^1(0,L))}^2=\Vert \widetilde w\widetilde w_{x}\Vert_{\L(0,\tau;H^1(0,L))}^2+\Vert( \widetilde w\widetilde w_{x})_t\Vert_{\L(0,\tau;H^1(0,L))}^2\\
    \;\leq\;  C(\tau^{\frac{1}{2}}+\tau^{\frac{1}{3}})^2\Vert\widetilde w\Vert_{S^3(\tau)}^4.
\end{equation}
As a consequence, from \eqref{pose1}, \eqref{pose2} and \eqref{pose3}, there exists a positive constant $C>0$ such that
\begin{equation}
    \Vert \cT_1(\widetilde w)\Vert_{S^3(\tau)}^2\leq C\left(\Vert \widetilde w_0\Vert_{H^3(0,L)}+ k^2\tau\vert \bar\eta_0\vert^2 + \left((\tau^{\frac{1}{2}}+\tau^{\frac{1}{3}})^2\Vert\widetilde w\Vert_{S^3(\tau)}^2+\left(\frac{2}{L}+L\right)(k^2\tau+k^2)\right)\Vert\widetilde w\Vert_{S^3(\tau)}^2  \right).
\end{equation}
Moreover, from \eqref{pose1}, \eqref{pose2} and \eqref{pose3}, we obtain  
\begin{equation*}
\begin{array}{l}
     \Vert \cT_1(\widetilde w^1)-\cT_1(\widetilde w^2)\Vert_{S^3(\tau)}^2\leq  
     \\[.5em] 
     \qquad \qquad \leq C\bigg( k^2  \int^\tau_0\left\vert \int^t_0 ( \widetilde w^1_x(s,0)-w^2_x(s,0))ds \right\vert^2 dt +\Vert w^1_x(\cdot,0)-w^2_x(\cdot,0)\Vert_{\L(0,\tau)}^2\   \\[.5em]
      \hfill +\Vert \widetilde w^2(\widetilde w^2_{x}-\widetilde w^1_{x})+\widetilde w^1_x(\widetilde w^2-\widetilde w^1)\Vert_{H^{1}(0,\tau;H^1(0,L))}^2 \bigg)\\
  \qquad \qquad     \leq  C\left(\!\left(\tfrac{2}{L}+L\right)(k^2\tau+k^2)+(\tau^{\frac{1}{2}}+\tau^{\frac{1}{3}})^2\Vert\widetilde w^1\Vert_{S^3(\tau)}^2+(\tau^{\frac{1}{2}}+\tau^{\frac{1}{3}})^2\Vert\widetilde w^2\Vert_{S^3(\tau)}^2\right)\Vert\widetilde w^1-\widetilde w^2\Vert_{S^3(\tau)}^2
    \end{array}    
\end{equation*}
for all $\widetilde w^1, \widetilde w^2 \in S^3(\tau)$.
We consider $\cT_1$ restricted to the closed ball $B_\rho=\{\widetilde w\in S^3(\tau): \Vert\widetilde w\Vert_{S^3(\tau)}\leq \rho\} \subset S^3(\tau)$ with $\rho$ to be chosen later. Then 
$$
     \Vert \cT_1(\widetilde w)\Vert_{S^3(\tau)}^2\leq C\left(\Vert \widetilde w_0\Vert_{H^3(0,L)}^2+ k^2\tau\vert \tilde\eta_0\vert^2 
     +\rho^2\left(\tfrac{2}L+L\right)k^2
          + \left((\tau^{\frac{1}{2}}+\tau^{\frac{1}{3}})^2\rho^2+
          \left(\tfrac{2}{L}+L\right)k^2\tau\right)\rho^2  \right)
$$ 
and 
$$
    \Vert \cT_1(\widetilde w^1)-\cT_1(\widetilde w^2)\Vert_{S^3(\tau)}^2
    \leq C\left(
    k^2 \left(\tfrac{2}{L}+L\right) +
    \left(\tfrac{2}{L}+L\right)k^2\tau+2(\tau^{\frac{1}{2}}+\tau^{\frac{1}{3}})^2\rho^2\right)\Vert\widetilde w^1-\widetilde w^2\Vert_{S^3(\tau)}^2.
$$ 
Finally, we select the constant $\rho, k_1$ and $\tau$
so that to obtain a contractive operator.
For instance, we can select $\rho = \sqrt{3C}\Vert \widetilde w_0\Vert_{H^3(0,L)}$ 
and 
$$
k_1^\star = \sqrt{\dfrac{1}{6C}\left(\dfrac{L}{2+L^2}\right)}
$$
and $\tau>0$  such that 
the following inequalities are satisfied
\begin{align*}
    \tau (k_1^\star)^2\vert \tilde\eta_0\vert^2
&< \Vert \widetilde w_0\Vert^2_{H^3(0,L)} \,,
\\
 \left((\tau^{\frac{1}{2}}+\tau^{\frac{1}{3}})^2\rho^2+\left(\tfrac{2}{L}+L\right)(k_1^\star)^2\tau\right) \rho^2 &< \tfrac12\Vert \widetilde w_0\Vert^2_{H^3(0,L)} \,,
 \\
  \left(\tfrac{2}{L}+L\right)(k_1^\star)^2\tau+2(\tau^{\frac{1}{2}}+\tau^{\frac{1}{3}})^2\rho^2& <\tfrac{1}2\,.
  \end{align*}
It follows that, for any $k\in (0,k_1^\star]$, 
$ \Vert \cT_1(\widetilde w)\Vert_{S^3(\tau)}\leq \rho$
for any $\widetilde w \in B_\rho$
and 
$
\Vert \cT_1(\widetilde w^1)-\cT_1(\widetilde w^2)\Vert_{S^3(\tau)}<\Vert\widetilde w^1-\widetilde w^2\Vert_{S^3(\tau)}
$ for any $\widetilde w^1, \widetilde w^2 \in B_\rho$.
Then, $\cT_1$  is a contraction operator from $B_\rho$  to $B_\rho$. According to the Banach fixed-point theorem, $\cT_1$ admits a unique fixed point.    Its unique fixed point is the desired solution
of \eqref{sys_reg_closed_loop_non1} for $0\leq t\leq \tau$. 
This shows that $\widetilde w$
in \eqref{sys_reg_closed_loop_non2} 
has a unique solution in $S^3(\tau)$.
Since $\widetilde w_x(t,0)$ is continuous on $[0,\tau]$ then 
$\tilde \eta$ is in $C^1(\tau)$
by definition of solution of an 
ODE. This concludes the proof of Lemma~\eqref{theo:well-posedness_nln}.
\end{proof}

\medskip

Note that we  have established the existence of unique classical solution of \eqref{sys_reg_closed_loop_non} locally in time. However, the Lyapunov functional introduced in the 
Section~\ref{sec_reg_linear} needed to establish Lemma~\ref{theo:well-posedness}
and Theorem~\ref{regulation} 
can be used to deduce the existence of unique solution global in time.  Indeed, since the time derivative of the Lyapunov functional will be proved to be non-increasing, this shows that the solution cannot explode for large time, proving thus that the solution exists for any positive time, 
as soon as the initial conditions are small enough.
The next result deals with the local exponential stability of 
the origin of system \eqref{sys_reg_closed_loop_non}.

\medskip

\begin{theorem}[Local Exponential Stability]\label{regulation_nln}
There exist  positive real number 
$k_2^*, \overbar w_\infty$,
 such that, 
for any $k\in (0,k_2^*)$  there exist positive real numbers
$\Delta,\nu_1,b_1$, such that 
for any solution
 to system \eqref{sys_reg_closed_loop_non}
with $w_\infty$ satisfying 
$\Vert w_\infty\Vert_{H^3}\leq \overbar w_\infty$ 
and  initial conditions 
$(\tilde\eta_0,\widetilde w_0)\in D(\mathcal{A})$ satisfying
$|\tilde \eta_0| +  \Vert \widetilde w_0\Vert_{\L}\leq 2\Delta$,
the following inequality holds
$    \Vert (\tilde \eta(t), \widetilde w(t)) \Vert_X\leq b_1e^{-\nu_1 t}\Vert (\tilde\eta_0,\widetilde w_0)\Vert_X$
for all $t\geq0$.
 Moreover
the regulation objective defined in 
\eqref{eq:regulation_objective2}
is satisfied.
\end{theorem}

\medskip

\begin{proof} The main idea of this proof is to extend the analysis developed
for the linear KdV model in Section~\ref{sec_reg_linear}.
In particular, following the main steps of the 
proof of Lemma~\ref{theo:well-posedness}, we aim at 
building a Lyapunov functional for the overall 
closed-loop system 
 \eqref{sys_reg_closed_loop_nln}
 by relying on Corollary~\ref{corollary:ISS-nonlineaire}. 
 Indeed, setting $a =- w_\infty'$, $b=-w_\infty$ and $d_2=k\eta$,
 system \eqref{sys_reg_closed_loop_non}
 is in the form 
 \eqref{eq_kdv_nonlinear_perturbed}.
 As a consequence, 
 there exist $\delta>0$ and a Lyapunov functional $V$ such that, 
 for any   $\Vert w_\infty'\Vert_{\infty}\leq \bar a$
 and $\Vert w_\infty \Vert_{W^{1,\infty}}\leq\bar b$, 
 the derivative of $V$ along the trajectories 
 of  system \eqref{sys_reg_closed_loop_non}
 satisfies
 \begin{equation}
 \label{eq:ISS-reg}
    \dot{V}(\widetilde w) \leq - \alpha \Vert\widetilde w\Vert^2_\L + 
    \sigma_1 k^2\tilde{\eta}^2 \qquad 
  \forall \, (\tilde \eta,\widetilde w)\in D_\delta(\cA)\,,
 \end{equation}
  with  $D_\delta(\cA) := \{(\tilde \eta,\widetilde w)\in D(\cA) : \Vert (\tilde \eta,\widetilde w)\Vert_{X}\leq \delta\}$.
 Now, we consider the Lyapunov functional $\VV$ defined in \eqref{lyapunov}. We want to show the local exponential stability of the origin of the 
system  \eqref{sys_reg_closed_loop_non}
with the functional $\VV$.
First, note that $ \VV$ 
is uniformly bounded by the norm in the space $X$
of $(\tilde \eta, \widetilde w)$, similar to 
inequality \eqref{ve}. 
Then, using \eqref{eq:ISS-reg}, its derivative along solutions to \eqref{sys_reg_closed_loop_nln} is given by, for any $(\widetilde \eta, \widetilde w)\in D_\delta(\cA)$.
 \begin{align}
 \label{eq:lyap-regV}
    \dot{\VV}(\widetilde \eta,\widetilde w)\leq - \alpha \Vert\widetilde w\Vert^2_\L + \sigma_1 k^2\widetilde\eta^2 + 2 F(\tilde \eta, \widetilde w) 
\end{align}
with $F(\tilde \eta, \widetilde w)  := \big(\tilde\eta-\cM \widetilde w\big)\big(\dot{\tilde\eta}-\cM\widetilde w_t\big)$.
After some integrations by parts, and recalling the property of $\cM$ in \eqref{ma}, we obtain
$$
    \dot{\tilde\eta}-\cM\tilde w_t=-k\tilde \eta(t) +\int^L_0 M(x)\widetilde w(x)w_\infty'(x) dx-\frac{1}{2}\int^L_0 M'(x)\widetilde w(x)^2 dx- \int^L_0 \big(M(x)w_\infty(x)\big)'\widetilde w(x) dx,
$$
from which we obtain
$$
F(\tilde \eta, w)  
= -k\tilde\eta^2 +\tilde\eta
k\cM \widetilde w + \tilde \eta\, \Phi(\widetilde w) - \cM \widetilde w \,\Phi(\widetilde w)
$$
with 
$$
    \Phi(\widetilde w) = \int^L_0 M(x)\widetilde w(x)w_\infty'(x) dx-\frac{1}{2}\int^L_0 M'(x)\widetilde w(x)^2 dx
    - \int^L_0 \big(M(x)w_\infty(x)\big)'\widetilde w(x) dx
$$
According to \eqref{M}, $M\in C^\infty([0,L])$. Therefore $M'$ is bounded on $[0,L]$. Then, using first Cauchy-Schwarz's inequality and then Young's inequality, we bound the terms   in $F$
as follows:
\begin{align*}
    \left\vert k\widetilde\eta\mathcal{M}\widetilde w\right\vert
 & \leq \frac{2k^2\Vert M\Vert^2_\L}{\alpha}\widetilde\eta^2 + \frac{\alpha}{8}\Vert \widetilde w\Vert^2_{\L} \,,
 \\
  |\tilde \eta  \Phi(\widetilde w)| 
& \leq  \frac{k}{2}\tilde\eta^2+\frac{1}{2k}
 \bigg(4\Vert Mw_\infty'\Vert_\L^2+4\Vert (Mw_\infty)'\Vert_\L^2+\Vert M'\Vert_\infty^2\Vert \widetilde w\Vert_\L^2\bigg)
 \Vert \widetilde w\Vert_\L^2 \,,
 \\
 |\cM\widetilde w \Phi(\widetilde w)| 
&\leq  \frac{\alpha}{8}\Vert \widetilde w\Vert^2_\L+\frac{2\Vert M\Vert^2_\L}{\alpha}\bigg(4\Vert Mw_\infty'\Vert_\L^2+4\Vert (Mw_\infty)'\Vert_\L^2+\Vert M'\Vert_\infty^2\Vert \widetilde w\Vert_\L^2\bigg)\Vert \widetilde w\Vert_\L^2\,.
\end{align*}
As a consequence, combining the previous bounds we further obtain
\begin{align*}
F(\tilde \eta, \widetilde w)\leq & \left(\frac{\alpha}{4}+\left(\frac{ 2\Vert M\Vert_\L^2}{\alpha}+\frac{1}{2k}\right)\bigg(4\Vert Mw_\infty'\Vert_\L^2+4\Vert (Mw_\infty)'\Vert_\L^2+\Vert M'\Vert_\infty^2\Vert \widetilde w\Vert_\L^2\bigg)\right)\Vert \widetilde w\Vert_\L^2\\
    &-\frac{k}{2}\left(1- 4k\frac{\Vert M\Vert_\L^2}{\alpha}\right)\widetilde\eta
    ^2\\ 
    \leq & \left(\frac{\alpha}{4}+\left(\frac{2 \Vert M\Vert_\L^2}{\alpha}+\frac{1}{2k}\right)\bigg(8\Vert M\Vert_{W^{1,\infty}}^2\Vert w_\infty\Vert_{H^3}^2+\Vert M'\Vert_\infty^2\Vert \widetilde w\Vert_\L^2\bigg)\right)\Vert \widetilde w\Vert_\L^2\\
    &-\frac{k}{2}\left(1- 4k\frac{\Vert M\Vert_\L^2}{a_1}\right)\widetilde \eta
    ^2
\end{align*} 
where, in the second inequality, we have used  $\Vert Mw_\infty'\Vert_\L^2+\Vert (Mw_\infty)'\Vert_\L^2\leq 2\Vert M\Vert_{W^{1,\infty}}^2\Vert w_\infty\Vert_{H^3}^2$.
Using the previous inequality together with \eqref{eq:lyap-regV} yields
 \begin{align}
    \dot{\mathcal{V}}(\widetilde \eta,\widetilde w)\leq &\left(-\frac{\alpha}{2}+\left(\frac{4 \Vert M\Vert_\L^2}{\alpha}+\frac{1}{k}\right)\bigg(8\Vert M\Vert_{W^{1,\infty}}^2\Vert w_\infty\Vert_{H^3}^2+\Vert M'\Vert_\infty^2\Vert \widetilde w\Vert_\L^2\bigg)\right)\Vert \widetilde w\Vert_\L^2\\\notag
    &+\left (\sigma_1 k^2-k\left(1- 4k\frac{\Vert M\Vert_\L^2}{\alpha}\right)\right)\widetilde \eta ^2.
\end{align}  
for all $(\widetilde\eta,\widetilde w)\in D_\delta(\mathcal{A})$. As a consequence, we can finally fix all the parameters.
In particular, we select 
$$
k^\star_2 = \min\left\{ 
k_0^\star, \; k_1^\star ,  \; 
\left(\dfrac{\alpha}{\alpha \sigma_1+4\Vert M\Vert_\L^2}   \right) 
\right\} 
$$
 with $k_0^\star$ given by Lemma~\ref{theo:well-posedness},
 and $k_1^\star$
 given by Lemma~\ref{theo:well-posedness_nln}.
Moreover, 
given any $k\in (0,k_2^\star)$, select  
$$
\overbar w_\infty= \min\left\{\bar a, \bar b, 
 \sqrt{\dfrac{\alpha}{64\Vert M\Vert_{W^{1,\infty}}^2\left(\frac{ 4\Vert M\Vert_\L^2}{\alpha}+\frac{1}{k}\right)}} \right\},
 $$
 with $\bar a, \bar b$
 given by Corollary~\ref{corollary:ISS-nonlineaire}.
 Moreover, let us define $\overbar \delta \in(0,\delta)$ satisfying
$$ \bar \delta^2\leq \dfrac{\alpha}{8\Vert M'\Vert_\infty^2}\left(\frac{ 4\Vert M\Vert_\L^2}{\alpha}+\frac{1}{k}\right)^{-1}.
$$ % The main idea of this proof is to extend the analysis developed
Using all these bounds we can finally conclude the existence of a positive constant $\varepsilon$ such that 
 \begin{equation}\label{eq:nonlinear_exp_stab_lyap}
     \dot {\VV} \leq  -\varepsilon(\Vert \widetilde w\Vert_\L^2  + \widetilde\eta^2)
     \qquad
     \forall (\tilde \eta, \widetilde w)\in D_{\bar\delta}(\cA)
 \end{equation}
 Finally, standard Lyapunov arguments briefly recalled
 here allows to conclude the result of the proof.
 In particular, consider a $c>0$ small 
 enough such that 
$\Omega_c := \{(\tilde \eta, \widetilde w)\in D(\cA) : \VV(\tilde \eta, \widetilde w) \leq c\} \subset D_{\bar\delta}(\cA)$.
Now consider any solution to \eqref{sys_reg_closed_loop_nln}
starting inside $\Omega_c$.
By Lemma~\ref{lemmanonlinear}
there exists $\tau>0$ such that 
such a solution exists on $[0,\tau]$. 
Let $T\geq\tau$ be its maximal interval of time of existence.
In view of \eqref{eq:nonlinear_exp_stab_lyap}, 
the derivative of $\VV$ is always negative, 
showing that such the solution cannot escape the 
level set $\Omega_c$. Hence, 
its maximal interval of existence is  $[0,\infty)$.
Moreover,  we can conclude the existence of a positive constant
$\Delta\in (0,\bar\delta)$ such that the set 
$D_{\Delta}(\cA)$ is included in the domain of attraction of 
the origin of system \eqref{sys_reg_closed_loop_non}.
Combining the Fr\'echet derivative  
\eqref{eq:frechet} and the Gr\"onwall's lemma 
with \eqref{eq:nonlinear_exp_stab_lyap} one can show the first part of the statement, that is $\Vert (\tilde \eta(t), \widetilde w(t)) \Vert_X\leq b_1e^{-\nu_1 t}\Vert (\tilde\eta_0,\widetilde w_0)\Vert_X$, for all $t\geq 0$ and for all $(\widetilde\eta_0,\widetilde w_0)\in D_\Delta(\cA)$.

Finally, to prove the second part of the statement, 
we can use the same argument as in the proof of \cite[Proposition 3.9]{rosier2006global} to deduce that there exists a continuous nonnegative function $\chi : \mathbb R_+\rightarrow \mathbb{R}_+$ and positive constants $C,\mu$ such that, for all $(\widetilde \eta_0,\widetilde w_0)\in D_\Delta(\cA)$ 
\begin{equation}\label{barwt1}
    \Vert \widetilde w_t(t,\cdot)\Vert_\L\leq Ce^{-\mu t} \chi(\Vert \widetilde w_0\Vert_\L) \Vert \widetilde w_t(0,\cdot)\Vert_\L \,,\qquad \forall t\geq 0\,.
\end{equation}
By  multiplying  the first equation of \eqref{sys_reg_closed_loop_non} by $\widetilde w$ and integrating by parts, we get after some computations
\begin{align*}
\notag
   - k^2\tilde \eta(t)^2+\widetilde w_x(t,0)^2&=-2\int^L_0 \widetilde w(t,x)\widetilde w_t(t,x)dx+\int^L_0\vert\widetilde w(t,x)\vert^2 w_\infty'(x)dx\\
   &\leq 2\Vert\widetilde w(t,\cdot)\Vert_\L\Vert\widetilde w_t(t,\cdot)\Vert_\L + \ell \Vert\widetilde w(t,\cdot)\Vert_\L^2\Vert\widetilde w_\infty\Vert_{H^3(0,L)}.
\end{align*} 
where $\ell$ is the constant of the embedding of $H^3(0,L)$ in $L^\infty(0,L)$. Then,  we can deduce 
$    \lim_{t\to\infty}\vert \widetilde w_x(t,0)\vert= 
     \lim_{t\to\infty}\vert w_x(t,0)-r\vert
 =    0 $
for all $(\tilde \eta_0,\widetilde w_0)\in D_\Delta(\cA)$, concluding the proof.
\end{proof}

Finally, 
by combining the statement 
of Lemma~\ref{lemmanonlinear}
and Theorem~\ref{regulation_nln} we have the following output regulation result
for the  system \eqref{sys_reg_closed_loop_nln}
in the original coordinates $w,\eta$.
The proof is omitted for space reasons.

\medskip
\begin{corollary}[Output Regulation]
There exist  positive real number $k_2^*, \bar d, \bar r$,
 such that, 
 for any disturbance $d$ and reference $r$
 satisfying $\Vert d\Vert_\L\leq \bar d$ and $|r|\leq \bar r$
 and for any $k\in (0,k_2^*)$  
 the output $y$ is asymptotically regulated at the reference $r$, namely 
 \eqref{eq:regulation_objective} is satisfied,
for any solution to system \eqref{sys_reg_closed_loop_nln},
with initial conditions
$(\eta_0,w_0)\in D(\cA)$  sufficiently small
in the norm $\RR\times \L(0,L)$.
\end{corollary}

\medskip

\color{black}
\section{Conclusion}
\label{sec:con}
In this article, we have solved the output regulation problem
by means of an integral action for a Korteweg-de-Vries (KdV) equation 
controlled at the boundary and 
subject to some distributed constant disturbances
 so that to regulate a boundary output 
to a given constant reference.
 For this, we have followed a Lyapunov approach. We have first designed an ISS Lyapunov functional which is obtained by strictifying the energy associated to the system. In particular, the energy is modified by adding a second term which is obtained from the design of an observer built with the backstepping technique. Then, thanks to this ISS Lyapunov functional, we have applied the forwarding method to achieve our goal in the context of output regulation.
In particular, we extended the system with an integral action 
and we designed an output feedback controller acting at the boundary.
We show that if the selected gain is sufficiently small then 
the solutions of the closed-loop system converge to an equilibrium. 
Furthermore, for  strong solutions, point-wise convergence of the
regulated output is achieved. Similar results hold locally for the nonlinear
model of the KdV, namely in presence of small references and perturbations  and with a local domain of attraction.\color{black}

For future work, we wish to { study the case in which} $L\in\mathcal{N}$. We believe furthermore that the
proposed strictification approach can be extended also to other classes of PDEs for which a strict Lyapunov functional is not yet known.

\section*{Acknowledgement}
We thank Eduardo Cerpa and Vincent Andrieu for the
fruitful discussions and precious suggestions. This research was partially supported by the French Grant ANR ODISSE (ANR-19-CE48-0004-01) and was also conducted in the framework of the regional programme "Atlanstic 2020, Research, Education and Innovation in Pays de la Loire'', supported
by the French Region Pays de la Loire and the European Regional Development Fund.

 \color{black}

\bibliographystyle{plain}
\bibliography{references}

\begin{thebibliography}{10}

\bibitem{astolfi2021repetitive}
D.~Astolfi, S.~Marx, and N.~van~de Wouw.
\newblock Repetitive control design based on forwarding for nonlinear
  minimum-phase systems.
\newblock {\em Automatica}, 129:109671, 2021.

\bibitem{astolfi2017integral}
D.~Astolfi and L.~Praly.
\newblock Integral action in output feedback for multi-input multi-output
  nonlinear systems.
\newblock {\em IEEE Transactions on Automatic Control}, 62(4):1559--1574, 2017.

\bibitem{bastin2021input}
G.~Bastin, J-M. Coron, and A.~Hayat.
\newblock Input-to-state stability in sup norms for hyperbolic systems with
  boundary disturbances.
\newblock {\em Nonlinear Analysis}, 208:112300, 2021.

\bibitem{bastin2015stability}
G.~Bastin, J.-M. Coron, and S.~O. Tamasoiu.
\newblock Stability of linear density-flow hyperbolic systems under {PI}
  boundary control.
\newblock {\em Automatica}, 53:37--42, 2015.

\bibitem{bona2003nonhomogeneous}
JL. Bona, SM. Sun, and BY. Zhang.
\newblock A nonhomogeneous boundary-value problem for the korteweg--de vries
  equation posed on a finite domain.
\newblock {\em Communications in Partial Differential Equations},
  28(7-8):1391--1436, 2003.

\bibitem{brezis(1973)}
H.~Brezis.
\newblock {\em Op\'erateurs maximaux monotones et semi-groupes de contractions
  dans les espaces de {H}ilbert}.
\newblock North-Holland, 1973.

\bibitem{brezis2010functional}
H.~Brezis.
\newblock {\em Functional analysis, Sobolev spaces and partial differential
  equations}.
\newblock Springer Science \& Business Media, 2010.

\bibitem{cerpa2014}
E.~Cerpa.
\newblock Control of a {Korteweg-de Vries} equation: a tutorial.
\newblock {\em Mathematical Control and Related Fields}, 4(1):45, 2014.

\bibitem{cerpacoron}
E.~Cerpa and J.M. Coron.
\newblock Rapid stabilization for a {Korteweg-de Vries} equation from the left
  dirichlet boundary condition.
\newblock {\em IEEE Transactions on Automatic Control}, 58(7):1688--1695, 2013.

\bibitem{chapouly}
M.~Chapouly.
\newblock Global controllability of a nonlinear {Korteweg--de Vries} equation.
\newblock {\em Communications in Contemporary Mathematics}, 11(03):495--521,
  2009.

\bibitem{cheverry2019handbook}
C.~Cheverry and N.~Raymond.
\newblock Handbook of spectral theory.
\newblock 2019.

\bibitem{chu2015asymptotic}
J.~Chu, J-M. Coron, and P.~Shang.
\newblock Asymptotic stability of a nonlinear {Korteweg--de Vries} equation
  with critical lengths.
\newblock {\em Journal of Differential Equations}, 259(8):4045--4085, 2015.

\bibitem{coron2020small}
J-M. Coron, Koenig A., and H-M. Nguyen.
\newblock On the small-time local controllability of a {KdV} system for
  critical lengths.
\newblock {\em arXiv preprint arXiv:2010.04478}, 2020.

\bibitem{lucoron}
J-M. Coron and Q.~L{\"u}.
\newblock Local rapid stabilization for a {Korteweg--de Vries} equation with a
  {N}eumann boundary control on the right.
\newblock {\em Journal de Math{\'e}matiques Pures et Appliqu{\'e}es},
  102(6):1080--1120, 2014.

\bibitem{coron2015feedback}
J.-M. Coron and S.~O. Tamasoiu.
\newblock Feedback stabilization for a scalar conservation law with {PID}
  boundary control.
\newblock {\em Chinese Annals of Mathematics, Series B}, 36(5):763--776, 2015.

\bibitem{corcre}
E~Cr{\'e}peau and J-M. Coron.
\newblock Exact boundary controllability of a nonlinear {KdV} equation with a
  critical length.
\newblock {\em J. Eur. Math. Soc}, 6:367--398, 2005.

\bibitem{curtain2020introduction}
R.~Curtain and H.~Zwart.
\newblock {\em Introduction to infinite-dimensional systems theory: a
  state-space approach}, volume~71.
\newblock Springer Nature, 2020.

\bibitem{de2003boundary}
J.~de~Halleux, C.~Prieur, J.-M. Coron, B.~d'Andr{\'e}a Novel, and G.~Bastin.
\newblock Boundary feedback control in networks of open channels.
\newblock {\em Automatica}, 39(8):1365--1376, 2003.

\bibitem{deutscher2017finite}
J.~Deutscher.
\newblock Finite-time output regulation for linear 2$\times$ 2 hyperbolic
  systems using backstepping.
\newblock {\em Automatica}, 75:54--62, 2017.

\bibitem{deutscher2017output}
J.~Deutscher.
\newblock Output regulation for general linear heterodirectional hyperbolic
  systems with spatially-varying coefficients.
\newblock {\em Automatica}, 85:34--42, 2017.

\bibitem{dos2008boundary}
V.~Dos~Santos, G.~Bastin, J.-M. Coron, and B.~d'Andr{\'e}a Novel.
\newblock Boundary control with integral action for hyperbolic systems of
  conservation laws: {Stability} and experiments.
\newblock {\em Automatica}, 44(5):1310--1318, 2008.

\bibitem{francis1975internal}
BA. Francis and WM. Wonham.
\newblock The internal model principle for linear multivariable regulators.
\newblock {\em Applied mathematics and optimization}, 2(2):170--194, 1975.

\bibitem{gagnon2020fredholm}
L.~Gagnon, P.~Lissy, and S.~Marx.
\newblock A fredholm transformation for the rapid stabilization of a degenerate
  parabolic equation.
\newblock {\em to appear in SIAM Journal on Control and Optimization}, 2021.

\bibitem{giaccagli2021sufficient}
M.~Giaccagli, D.~Astolfi, V.~Andrieu, and L.~Marconi.
\newblock Sufficient conditions for global integral action via incremental
  forwarding for input-affine nonlinear systems.
\newblock {\em IEEE Transactions on Automatic Control}, 2021.

\bibitem{huhtala2021approximate}
K.~Huhtala and L.~Paunonen.
\newblock Approximate local output regulation for a class of nonlinear fluid
  flows.
\newblock {\em European Journal of Control}, 62:136--142, 2021.

\bibitem{jacob2020noncoercive}
B.~Jacob, A.~Mironchenko, JR. Partington, and F.~Wirth.
\newblock Noncoercive lyapunov functions for input-to-state stability of
  infinite-dimensional systems.
\newblock {\em SIAM Journal on Control and Optimization}, 58(5):2952--2978,
  2020.

\bibitem{kafnemer2021weak}
M.~Kafnemer, B.~Mebkhout, and Y.~Chitour.
\newblock Weak input to state estimates for 2d damped wave equations with
  localized and nonlinear damping.
\newblock {\em SIAM Journal on Control and Optimization}, 59(2):1604--1627,
  2021.

\bibitem{karafyllis2019input}
I.~Karafyllis and M.~Krstic.
\newblock {\em Input-to-state stability for PDEs}.
\newblock Springer, 2019.

\bibitem{lhachemi2020pi}
H.~Lhachemi, C.~Prieur, and E~Tr{\'e}lat.
\newblock Pi regulation of a reaction--diffusion equation with delayed boundary
  control.
\newblock {\em IEEE Transactions on Automatic Control}, 66(4):1573--1587, 2020.

\bibitem{liard2022boundary}
T.~Liard, I.~Balogoun, S.~Marx, and F.~Plestan.
\newblock Boundary sliding mode control of a system of linear hyperbolic
  equations: a lyapunov approach.
\newblock {\em Automatica}, 135:109964, 2022.

\bibitem{logemann1997low}
H.~Logemann and S.~Townley.
\newblock Low-gain control of uncertain regular linear systems.
\newblock {\em SIAM journal on control and optimization}, 35(1):78--116, 1997.

\bibitem{KrsticAndrey}
A.~Smyshlyaev M.~Krstic.
\newblock {\em Boundary control of {PDE}s: a course on backstepping designs}.
\newblock Society for Industrial and Applied Mathematic, 2008.

\bibitem{fmazenc}
M.~Malisoff and F.~Mazenc.
\newblock {\em Constructions of strict {L}yapunov functions}.
\newblock Springer Science and Business Media, 2009.

\bibitem{marx2017cone}
S.~Marx, V.~Andrieu, and C.~Prieur.
\newblock {Cone-bounded feedback laws for $m$-dissipative operators on Hilbert
  spaces}.
\newblock {\em Mathematics of Control, Signals, and Systems}, 29(4):1--32,
  2017.

\bibitem{marx2021forwarding}
S~Marx, L~Brivadis, and D~Astolfi.
\newblock Forwarding techniques for the global stabilization of dissipative
  infinite-dimensional systems coupled with an ode.
\newblock {\em Mathematics of Control, Signals, and Systems}, 33(4):755--774,
  2021.

\bibitem{marx2014output}
S.~Marx and E.~Cerpa.
\newblock {Output feedback control of the linear Korteweg-de Vries equation}.
\newblock In {\em 53rd IEEE conference on decision and control}, pages
  2083--2087. IEEE, 2014.

\bibitem{marxcerpa}
S.~Marx and E.~Cerpa.
\newblock Output feedback stabilization of the {Korteweg--de Vries} equation.
\newblock {\em Automatica}, 87:210--217, 2018.

\bibitem{marx2015stabilization}
S.~Marx, E.~Cerpa, C.~Prieur, and V.~Andrieu.
\newblock Stabilization of a linear {Korteweg-de Vries} equation with a
  saturated internal control.
\newblock In {\em 2015 European Control Conference (ECC)}, pages 867--872.
  IEEE, 2015.

\bibitem{swanncpa1}
S.~Marx, E.~Cerpa, C.~Prieur, and V.~Andrieu.
\newblock Global stabilization of a {Korteweg--De Vries} equation with
  saturating distributed control.
\newblock {\em SIAM Journal on Control and Optimization}, 55(3):1452--1480,
  2017.

\bibitem{marx2018stability}
S.~Marx, Y.~Chitour, and C.~Prieur.
\newblock Stability results for infinite-dimensional linear control systems
  subject to saturations.
\newblock In {\em 2018 European Control Conference (ECC)}, pages 2995--3000.
  IEEE, 2018.

\bibitem{marx2019stability}
S.~Marx, Y.~Chitour, and C.~Prieur.
\newblock Stability analysis of dissipative systems subject to nonlinear
  damping via {Lyapunov} techniques.
\newblock {\em IEEE Transactions on Automatic Control}, 65(5):2139--2146, 2019.

\bibitem{mazenc1996adding}
F.~Mazenc and L.~Praly.
\newblock Adding integrations, saturated controls, and stabilization for
  feedforward systems.
\newblock {\em IEEE Transactions on Automatic Control}, 41(11):1559--1578,
  1996.

\bibitem{prieur}
A.~Mironchenko and C.~Prieur.
\newblock Input-to-state stability of infinite-dimensional systems: recent
  results and open questions.
\newblock {\em SIAM Review}, 62(3):529--614, 2020.

\bibitem{natarajan2016approximate}
V.~Natarajan and J.~Bentsman.
\newblock Approximate local output regulation for nonlinear distributed
  parameter systems.
\newblock {\em Mathematics of Control, Signals, and Systems}, 28(3):1--44,
  2016.

\bibitem{paunonen2010internal}
L.~Paunonen and S.~Pohjolainen.
\newblock Internal model theory for distributed parameter systems.
\newblock {\em SIAM Journal on Control and Optimization}, 48(7):4753--4775,
  2010.

\bibitem{pazy}
A.~Pazy.
\newblock {\em Semigroups of Linear Operators and Applications to Partial
  Differential Equations}.
\newblock Applied Mathematical Sciences, 1983.

\bibitem{phong1991operator}
V.~Q. Ph{\'o}ng.
\newblock The operator equation ${AX}- {XB}= {C}$ with unbounded operators
  ${A}$ and ${B}$ and related abstract cauchy problems.
\newblock {\em Mathematische Zeitschrift}, 208(1):567--588, 1991.

\bibitem{pohjolainen1982robust}
S.~Pohjolainen.
\newblock Robust multivariable {PI}-controller for infinite dimensional
  systems.
\newblock {\em IEEE Transactions on Automatic Control}, 27(1):17--30, 1982.

\bibitem{pohjolainen1985robust}
S.~Pohjolainen.
\newblock Robust controller for systems with exponentially stable strongly
  continuous semigroups.
\newblock {\em Journal of mathematical analysis and applications},
  111(2):622--636, 1985.

\bibitem{praly2019}
L.~Praly.
\newblock Observers to the aid of ``strictification'' of {Lyapunov} functions.
\newblock {\em Systems and Control Letters}, 134:104510, 2019.

\bibitem{prieurmazenc}
C.~Prieur and F.~Mazenc.
\newblock {ISS}-{L}yapunov functions for time-varying hyperbolic systems of
  balance laws.
\newblock {\em Mathematics of Control, Signals, and Systems}, 24(1-2):111--134,
  2012.

\bibitem{rosier}
L.~Rosier.
\newblock Exact boundary controllability for the {Korteweg-de Vries} equation
  on a bounded domain.
\newblock {\em ESAIM: Control, Optimisation and Calculus of Variations},
  2:33--55, 1997.

\bibitem{rosier2006global}
L.~Rosier and B-Y. Zhang.
\newblock Global stabilization of the generalized korteweg--de vries equation
  posed on a finite domain.
\newblock {\em SIAM Journal on Control and Optimization}, 45(3):927--956, 2006.

\bibitem{tang2018asymptotic}
S.~Tang, J.~Chu, P.~Shang, and J-M. Coron.
\newblock Asymptotic stability of a {Korteweg--de Vries} equation with a
  two-dimensional center manifold.
\newblock {\em Advances in Nonlinear Analysis}, 7(4):497--515, 2018.

\bibitem{tang2015stabilization}
S.~Tang and M.~Krstic.
\newblock {Stabilization of linearized Korteweg-de Vries systems with
  anti-diffusion by boundary feedback with non-collocated observation}.
\newblock In {\em 2015 American Control Conference (ACC)}, pages 1959--1964.
  IEEE, 2015.

\bibitem{terrand2019adding}
A.~Terrand-Jeanne, V.~Andrieu, DS.~V. Martins, and CZ~Xu.
\newblock Adding integral action for open-loop exponentially stable semigroups
  and application to boundary control of pde systems.
\newblock {\em IEEE Transactions on Automatic Control}, 65(11):4481--4492,
  2019.

\bibitem{trinh2016multivariable}
N.-T Trinh, V.~Andrieu, and C.-Z. Xu.
\newblock Multivariable {PI} controller design for 2$\times$ 2 systems governed
  by hyperbolic partial differential equations with lyapunov techniques.
\newblock In {\em 2016 IEEE 55th Conference on Decision and Control (CDC)},
  pages 5654--5659. IEEE, 2016.

\bibitem{xu1995robust}
C.-Z. Xu and H.~Jerbi.
\newblock A robust {PI}-controller for infinite-dimensional systems.
\newblock {\em International Journal of Control}, 61(1):33--45, 1995.

\bibitem{xu2014multivariable}
C.-Z. Xu and G.~Sallet.
\newblock Multivariable boundary {PI} control and regulation of a fluid flow
  system.
\newblock {\em Mathematical Control and Related Fields}, 4(4):501--520, 2014.

\bibitem{zhang2020local}
L.~Zhang, C.~Prieur, and J.~Qiao.
\newblock Local proportional-integral boundary feedback stabilization for
  quasilinear hyperbolic systems of balance laws.
\newblock {\em SIAM Journal on Control and Optimization}, 58(4):2143--2170,
  2020.

\bibitem{zheng2018input}
J.~Zheng and G.~Zhu.
\newblock Input-to-state stability with respect to boundary disturbances for a
  class of semi-linear parabolic equations.
\newblock {\em Automatica}, 97:271--277, 2018.

\end{thebibliography}

\end{document}